\documentclass{article}
\usepackage{latexsym}
\usepackage{amssymb}
\usepackage{amsfonts}
\usepackage{amsmath}
\usepackage{amsthm}
\usepackage{ulem}	
\usepackage[all]{xy}	
\usepackage{bm}		
\usepackage{wasysym}	
\usepackage{ifthen}
\usepackage{graphics}
\usepackage{psfrag, psfig}
\usepackage[dvips]{graphicx}
\usepackage{h2-macro}

\title{Renormalisable H\'enon-like Maps and Unbounded Geometry}
\author{P.~E.~Hazard \and M.~Lyubich \and M.~Martens}
\date{\today}

\newcounter{bean}
\newcounter{beany}
\renewenvironment{enumerate}{\begin{list}{(\roman{bean})}{\usecounter{bean}\setlength{\rightmargin}{\leftmargin}}}{\end{list}\setcounter{bean}{1}}
\newenvironment{a-enumerate}{\begin{list}{(A-\arabic{beany})}{\usecounter{beany}\setlength{\rightmargin}{\leftmargin}}}{\end{list}}
\newenvironment{b-enumerate}{\begin{list}{(B-\arabic{bean})}{\usecounter{bean}\setlength{\rightmargin}{\leftmargin}}}{\end{list}\setcounter{bean}{1}}
\newenvironment{d-enumerate}{\begin{list}{(D-\arabic{bean})}{\usecounter{bean}\setlength{\rightmargin}{\leftmargin}}}{\end{list}\setcounter{bean}{1}}

\theoremstyle{plain}

\newtheorem{thm}{Theorem}[section]
\newtheorem{lem}[thm]{Lemma}
\newtheorem{prop}[thm]{Proposition}
\newtheorem{cor}[thm]{Corollary}

\theoremstyle{definition}
\newtheorem{defn}[thm]{Definition}

\theoremstyle{remark}
\newtheorem{rmk}[thm]{Remark}

\numberwithin{equation}{section}

\begin{document}
\maketitle

\def\IMSmarkvadjust{0 pt}
\def\IMSmarkhadjust{0 pt}
\def\IMSmarkhpadding{0 pt}
\def\IMSpubltext{Published in modified form:}
\def\SBIMSMark#1#2#3{
 \font\SBF=cmss10 at 10 true pt
 \font\SBI=cmssi10 at 10 true pt
 \setbox0=\hbox{\SBF \hbox to \IMSmarkhpadding{\relax}
                Stony Brook IMS Preprint \##1}
 \setbox2=\hbox to \wd0{\hfil \SBI #2}
 \setbox4=\hbox to \wd0{\hfil \SBI #3}
 \setbox6=\hbox to \wd0{\hss
             \vbox{\hsize=\wd0 \parskip=0pt \baselineskip=10 true pt
                   \copy0 \break%
                   \copy2 \break%
                   \copy4 \break}}
 \dimen0=\ht6   \advance\dimen0 by \vsize \advance\dimen0 by 8 true pt
                \advance\dimen0 by -\pagetotal
	        \advance\dimen0 by \IMSmarkvadjust
 \dimen2=\hsize \advance\dimen2 by .25 true in
	        \advance\dimen2 by \IMSmarkhadjust

%
%
  \openin2=publishd.tex
  \ifeof2\setbox0=\hbox to 0pt{}
  \else 
     \setbox0=\hbox to 3.1 true in{
                \vbox to \ht6{\hsize=3 true in \parskip=0pt  \noindent  
                {\SBI \IMSpubltext}\hfil\break
                \input publishd.tex 
                \vfill}}
  \fi
  \closein2
  \ht0=0pt \dp0=0pt
 \ht6=0pt \dp6=0pt
 \setbox8=\vbox to \dimen0{\vfill \hbox to \dimen2{\copy0 \hss \copy6}}
 \ht8=0pt \dp8=0pt \wd8=0pt
 \copy8
 \message{*** Stony Brook IMS Preprint #1, #2. #3 ***}
}

\SBIMSMark{2010/2}{February 2010}{}

\vspace{-1cm}

\begin{abstract}
We show that given a one parameter family $F_b$ of strongly dissipative infinitely renormalisable H\'enon-like maps, parametrised by a quantity called the `average Jacobian'
$b$, the set of all parameters $b$ such that $F_b$ has a Cantor set with unbounded geometry has full Lebesgue measure.
\end{abstract}
\section{Introduction}
\subsection{Background}
In~\cite{dCML} de Carvalho and two of the current authors constructed a period-doubling renormalisation theory for H\'enon-like mappings of the form
\begin{equation}
F(x,y)=(f(x)-\e(x,y),x).
\end{equation}
Here $f$ is a unimodal map and $\e$ is a real-valued map from the square to the positive real numbers of small size (we shall be more explicit about the maps under consideration in Section~\ref{prelim}). Their results were extended in~\cite{Haz1} to arbitrary stationary combinatorics. This paper picks up where~\cite{Haz1} left off, by considering the geometry of the invariant Cantor set $\Cantor$ of $F$, constructed in those two papers, in more detail. 

For a long time it was assumed that the properties satisfied by the one dimensional unimodal renormalisation theory would also be satisfied by any renormalisation theory in any dimension. In the above two papers this was shown to be false. More specifically it was shown that at a special point $\tau$ of the Cantor set $\Cantor$ the renormalisations converged at a universal rate for each stationary combinatorial type. It was also shown that any conjugacy between the Cantor sets $\Cantor$ and $\tilde\Cantor$ for two given infinitely renormalisable H\'enon-like maps $F$ and $\tilde F$ of the same combinatorial type, which preserves tips, can only be $C^1$ if the \emph{average Jacobians} of $F$ and $\tilde F$ are equal (see below and~\cite{Haz1} for more precise statements). Hence universality at the tip is not equivalent to rigidity at the tip.

Another aspect of the renormalisation theory for unimodal maps is the notion of \emph{a priori} bounds. These are uniform or eventually uniform bounds for the geometry of the images of
the central interval at each renormalisation step. More precisely, let $f\colon [0,1]\to [0,1]$ be a unimodal map with central intervals $I_{i+1}\subset I_{i}$ of levels $i+1$ and $i$ respectively. If $J=f^k(I_{i})$ and $J'=f^k(I_{i+1})$ (where $k>0$ is some integer less that the return time of $I_{i}$) then $|J'|/|J|, |L'|/|J|$ and $|R'|/|J|$ are (eventually) uniformly bounded from below. Here $L', R'$ are the left and right connected components of $J-J'$. It is on such properties that the current paper will concern itself. 

Several authors have worked on consequences of a similar notion of \emph{a priori} bounds in the two dimensional case. For example Catsigeras, Moreira and Gambaudo~\cite{CGM1} and Moreira~\cite{Mor1} consider common generalisations of the model introduced by Bowen, Franks and Young in~\cite{BandF} and~\cite{FandY}, and the model introduced by Gambaudo, Tresser and van Strien in~\cite{GST1}. The first paper,~\cite{CGM1}, shows that given a dissipative infinitely renormalisable diffeomorphism of the disk with bounded combinatorics and bounded geometry, there is a dichotomy: either it has positive topological entropy or it is eventually period doubling. In the second paper~\cite{Mor1}, a comparison is made between the smoothness and combinatorics of the two models using the asymptotic linking number: given a period doubling $C^\infty$, dissipative, infinitely renormalisable diffeomorphism of the disk with bounded geometry the convergents of the asymptotic linking number cannot converge monotonically. This should be viewed as a kind of combinatorial rigidity result which, in particular, implies that Bowen-Franks-Young maps cannot be $C^\infty$.

\paragraph{Acknowledgements.} The authors would like to thank Michael Benedicks for many useful discussions during his stay at Stony Brook in Spring 2008. We also thank Sebastian van Strien and Andr\'e de Carvalho for their insights on H\'enon dynamics and their many useful comments on the current work.
\subsection{Statement of Results}
However, we would like to note that \emph{as of yet} no example of an infinitely renormalisable H\'enon-like map with bounded geometry is known. To the authors knowledge, in the slightly more general case of infinitely renormalisable diffeomorphisms of the disk (considered in the above two papers), no example with bounded geometry is known either. In fact, at least for the H\'enon-like case, we will show the following result:  
\begin{thm}\label{main}
Let $F_b$ be a one parameter family, parametrised by the average Jacobian $b=b(F_b)\in [0,b_0)$, of infinitely renormalisable H\'enon-like maps. Then there is a subinterval $[0,b_1]\subset
[0,b_0)$ for which there exists a dense $G_\delta$ subset
$S\subset [0,b_1)$ with full relative Lebesgue measure such that the Cantor set $\Cantor(b)=\Cantor(F_b)$ has unbounded geometry for all $b\in S$.
\end{thm}
We now outline the structure of the paper. In the next section we will review the results of~\cite{Haz1} that will be necessary to prove the above theorem, with a quick primer on unimodal renormalisation theory to aid with setting our notations. In the following section we define \emph{boxings} of the Cantor set. These are nested sequences of pairwise disjoint simply connected domains that `nest down' to the Cantor set $\Cantor$ and are invariant under the dynamics. We then introduce our construction and the mechanism that will destroy the geometry of our boxings, namely \emph{horizontal overlapping}. Then we give a condition in terms of the average Jacobian for horizontal overlapping of boxes to occur. We show this condition is satisfied for a dense $G_\delta$ set of parameters with full Lebesgue measure. This last part is purely analytical and has no dynamical content. 
\subsection{Open Problems}
Before proceeding we would like to state some open problems suggested by the current work. As was mentioned above, the biggest problem appears to be whether any infinitely renormalisable H\'enon-like map has bounded geometry. This, however, would require different
machinery to that introduced~\cite{dCML} and ~\cite{Haz1}, or a least an extension of it. The difficulty lies in bounded geometry being a global property whereas,
only the local behaviour around the `tip' of the Cantor set is relatively well known. (However, recent work has shown the geometry can be well understood in a distibutional sense, see~\cite{LM3}). 

Specifically, we draw the readers attention to the dichotomy shown in the proof of Proposition~\ref{G-thm}. This states that if $A_1\sigma\geq A_0$ then there are no parameters giving bounded geometry Cantor sets, where $\sigma$ is the scaling ratio and $A_0, A_1$ are the constants from Proposition~\ref{parameterh-overlap}. The value of $\sigma$ is determined by the combinatorial type of the maps we are considering, whereas $A_0,A_1$ depend also upon the choice of well-chosen words and and well-placed points (see Section~\ref{construction} for definitions). Ultimately the admissable well-chosen words depend on the combinatorial type also, or more precisely on the structure of the presentation functions for that combinatorial type. This suggests it may be possible to show there is no bounded geometry, for any parameter values, in certain classes of combinatorial types. This would require a finer analysis of the one-dimensional presentation functions than is currently available.

A more preliminary step would also be to find the Hausdorff dimension of the set $S$ in Theorem~\ref{main}. This would simply be a further analysis of our construction of $S$, however it may be the case that, as in the previous problem, more control over the relative sizes of $A_0$ and $A_1$ will be required.  

\section{Preliminaries}\label{prelim}
\subsection{Notations and Conventions}
Let $\pi_x,\pi_y\colon \RR^2\to \RR$ denote the projections onto the $x-$ and $y-$ coordinates. We will identify these with their extensions to $\CC^2$. (In fact we will
identify all real functions with their complex extensions whenever they exist.) 

Given points $a,b\in \RR$ we will denote the closed interval between $a$ and $b$ by $[a,b]=[b,a]$.
Throughout we will denote the interval $[-1,1]$ by $J$ and the square $[-1,1]^2=J^2$ by $B$.

Let $\Omega\subset\CC^2$ be the product of two simply connected domains in $\CC$ compactly containing $B^2$. That is $\Omega=\Omega_x\times\Omega_y$ where
$\Omega_x=\pi_x(\Omega),\Omega_y=\pi_y(\Omega)\subset \CC$ are disks containing $J$.

Given points $z_0, z_1\in B$, the \emph{rectangle spanned by $z_0$ and $z_1$} is given by
\[\brl z_0,z_1\brr=[\pi_x(z_0),\pi_x(z_1)]\times[\pi_y(z_0),\pi_y(z_1)],\]
and the \emph{straight line segment between $z$ and $\tilde{z}$} is denoted by $[z,\tilde{z}]$. The convex hull of a set $S\subset \RR^2$ will be denoted by $\hull(S)$.

We say that two planar sets \emph{horizontally overlap} if they mutually intersect a vertical line, that is if their projections onto the $x$-axis intersect. Similarly we say
two planar sets \emph{vertically overlap} if they mutually intersect a horizontal line, which is equivalent to saying that their projections onto the $y$-axis intersect. 

We say two planar sets $S_0, S_1$ are \emph{horizontally separated} if $\pi_x(\hull(S_0))\cap\pi_x(\hull(S_1))=\emptyset$. Similarly we say the sets $S_0,S_1\subset\RR^2$
are \emph{vertically separated} if $\pi_y(\hull(S_0))\cap\pi_y(\hull(S_1))=\emptyset$.

Let $M, N$ be manifolds and $r=0,1\ldots,\infty,\omega$. We denote by $C^r(M,N)$ the space of $C^r$-maps from $M$ to $N$ and by $\Emb^r(M,N)$ the space of $C^r$-embeddings, that is, diffeomorphisms onto their images if $M$ and $N$ have the same dimension.

\subsection{Unimodal Maps}\label{subsect:unimodal}
Let $\U_{\Omega_x}$ denote the space of maps $f\in C^\omega(J,J)$ satisfying
\begin{enumerate}
\item $f$ has a unique critical point $c_0=c(f)$ which lies in $(-1,1)$;
\item $f$ is orientation preserving to the left of $c_0$ and orientation reversing to the right of $c_0$;
\item $J$ is the dynamical interval for $f$, that is, if $c_i=\o{i}{f}(c_0)$, then $c_1=1, c_2=-1$;
\item $f$ admits a holomorphic extension to the domain $\Omega_x$, upon which it can be factored as $\psi\circ Q\circ \ii$ where $\ii\colon J\to [-a,1]$ is the unique orientation
preserving affine bijection between those domains, $Q\colon\CC\to\CC$ is given by $Q(z)=1-z^2$ and
$\psi\colon Q\circ\ii(\Omega_x)\to\CC$ is univalent and fixes the real axis;
\item there is a unique expanding fixed point in the interior of $J$.
\end{enumerate}
Such maps\footnote{We will also assume critical points are uniformly bounded from the critical value. If this bound is sufficiently small a neighbourhood of the renormalisation
fixed point will be contained in this space.} will be called \emph{$\U$-maps}. We will identify all $\U$-maps with their holomorphic extensions. We make two observations: first, this extension will be
$\RR$-symmetric (i.e. $f(\bar z)=\overline{f(z)}$ for all $z\in\Omega_x$) and second, the expanding fixed point will have negative multiplier.

\begin{defn}[unimodal permutation]
Given a permutation $\upsilon$ of the set $W_p=\{0,1,\ldots,p-1\}\subset\RR$ we construct a map $g_\upsilon\colon [0,p-1]\to [0,p-1]$ by setting
\[g_\upsilon\colon x\mapsto \left\{\begin{array}{ll}\upsilon(x) & x\in W_p \\ \upsilon(i)+(x-i)(\upsilon(i+1)-\upsilon(i)) & x\in (i,i+1), i\in W_p \end{array}\right.\]
then extending affinely between these points.

A permutation $\upsilon$ of the set $W_p=\{0,1,\ldots,p-1\}$ is called a \emph{unimodal permutation} if $g_\upsilon$ has exactly two domains of
monotonicity, on the left one $g_\upsilon$ is increasing and on the right one is decreasing.
\end{defn}
\begin{defn}[renormalisable]
A map $f\in \U_{\Omega_x}$ is \emph{renormalisable with combinatorics $\upsilon$} if 
\begin{enumerate}
\item there is a subinterval $J^0\subset J$ containing the critical point such that $\o{p}{f}(J^0)\subset J^0$;
\item the interiors of the subintervals $J^i=\o{i}{f}(J^0), i\in W_p$ are pairwise disjoint;
\item $f$ acts on $\uline{J}=\{J^0,J^1=f(J^0),\ldots J^{p-1}=\o{p-1}{f}(J^0)\}$, embedded in the line with the standard orientation, as $\upsilon$ acts on the symbols $W_p=\{0,1,\ldots,p-1\}$. More precisely, if $J',J''\in\uline{J}$ are the $i$-th and $j$-th intervals from the left endpoint of $J$ respectively. Then $f(J')$ lies to the left of $f(J'')$ if and only if $\upsilon(i)<\upsilon(j)$;
\item the map 
\[\RU f=h^{-1}\circ \o{p}{f}\circ h\]
is an element of $\U_{\Omega_x}$ for an affine bijection $h$ from $J$ to $J^0$. Note there are exactly two such affine bijections, but there will only be one such that $\RU f\in \U_{\Omega_x}$;
\end{enumerate}
The map $\RU f$ is called the \emph{renormalisation of $f$} and the operator $\RU$ the \emph{renormalisation operator of combinatorial type $\perm$}. 

\end{defn}
Let $\U_{\Omega_x,\upsilon}$ denote the subspace consisting of maps $f\in\U_{\Omega_x}$ which are renormalisable of combinatorial type $\perm$.  
If $\RU^n f\in \U_{\Omega_x,\perm}$ for all $n\geq 0$ then we will say $f$ is \emph{infinitely renormalisable with stationary combinatorics $\perm$}. It will be these maps we are most interested in.

Henceforth we will fix a unimodal permutation $\perm$ and drop the $p$ from $W_p$. That is we denote $\{0,1,\ldots,p-1\}$ by $W$. We will maintain the $\perm$ in $\U_{\Omega_x,\perm}$ to distinguish it from the space of (possibly non-renormalisable) unimodal maps $\U_{\Omega_x}$. 
Let $W^n$ denote the set of all words over $W$ of length $n$, 
let $W^*$ denote the sets of words over $W$ of arbitrary finite length and 
let $\bar W$ denote the space of all words of infinite length. 
We endow $W^*$ and $\bar W$ with the structure of an adding machine and denote the transformation ``addition with infinite carry'' by $\word{w}{}\mapsto 1+\word{w}{}$. 
That is, for $\word{w}{}=w_0\ldots w_n\in W^*$,
\[\word{w}{}\mapsto 1+\word{w}{}=\left\{\begin{array}{ll}
(1+w_0,w_1,\ldots,w_n) & w_0\neq p-1 \\ 
(0,0,\ldots,0,1+w_k,\ldots,w_n) & w_0,\ldots,w_{k-1}=p-1, w_k\neq p-1 \\
(\underbrace{0,\ldots,0}_{n-times},1) & w_0,\ldots,w_n=p-1.
\end{array}\right.\]
The addition on $\bar W$ is similar. If $f\in\U_{\Omega_x,\upsilon}$ is infinitely renormalisable there is a collection $\uline{J}=\{J^{\word{w}{}}\}_{\word{w}{}\in W^*}$ of subintervals with the following properties:
\begin{enumerate}
\item $f(J^{\word{w}{}})=J^{1+\word{w}{}}$ for all $\word{w}{}\in W^*$;
\item $J^{\word{w}{}}$ and $J^{\word{\tilde w}{}}$ are disjoint for all $\word{w}{}\neq \word{\tilde w}{}$ of the same length;
\item the disjoint union of the $J^{\word{w}{w}}$, $w\in W$, is a subset of $J^{\word{w}{}}$, for all $\word{w}{}\in W^*$.
\end{enumerate} 
The following is integral to the renormalisation theory of unimodal maps. (See~\cite{dMvS} for the proof and more details.)
\begin{thm}[Real A Priori Bounds]\label{real-ap-bounds}
Let $f\in\U_{\Omega_x,\upsilon}$ be an infinitely renormalisable unimodal map. Then there exists constants $L>1$ and $0<k_0<k_1<1$ such that 
\begin{enumerate}
\item $L^{-1}<|J^{\word{w}{w}}|/|J^{\word{w}{\tilde{w}}}|<L$;
\item $k_0<|J^{\word{w}{w}}|/|J^{\word{w}{}}|<k_1$;
\end{enumerate}
for all $\word{w}{}\in W^*, w,\tilde{w}\in W$. Moreover these bounds are \emph{beau}\footnote{This means there are universal $L, k_i$, depending upon $\upsilon$ only such that
the following holds: given any infinitely renormalisable
$f$ there is an $N>0$ such that these bounds hold for all $\word{w}{}\in W^n$, $n>N$. According to Sullivan, beau stands for ``bounded eventually and universally''.}.
\end{thm}
\begin{thm}[Existence and Uniqueness of the Fixed Point]\label{unimodal-fix-point}
For any unimodal permutation $\upsilon$ there exists a unique $\RU$-fixed point $f_*=f_{*,\upsilon}\in \U_{\Omega_x}$, i.e.
\[\RU f_*=f_*.\]
\end{thm}
\begin{thm}[Hyperbolicity of the Fixed Point]\label{unimodal-hyperbolicity}
For any $\upsilon$ the fixed point $f_*$ is hyperbolic with a codimension one stable manifold in $\U_{\Omega_x}$.
\end{thm}

\subsection{H\'enon-like Maps}\label{subsect:henonlike}
Let $\bar\e>0$. Let $\T_{\Omega}(\bar\e)$ denote the space of maps $\e\in C^\omega(B,\RR_{\geq 0})$, which satisfy 
\begin{enumerate}
\item $\e(x,0)=0$;
\item $\del_y\e\neq 0$;
\item $\e$ admits a holomorphic extension to $\Omega$;
\item $|\e|_{\Omega}\leq \bar\e$, where $|\!-\!|_{\Omega}$ denotes the sup-norm on $\Omega$. 
\end{enumerate}
Such maps will be called \emph{thickenings} or \emph{$\bar\e$-thickenings} if we want to emphasise it's thickness $\bar\e>0$.

Let $\H_{\Omega}(\bar\e)$ denote the space of diffeomorphisms onto their images, $F\in \Emb^\omega(B,\RR^2)$, admitting a holomorphic extension to $\Omega$, expressible as $F=(f\circ\pi_x-\e,\pi_x)$ where $f\in \U_{\Omegax}$ and
$\e\in\T_{\Omega}(\bar\e)$.
Such maps will be called \emph{parametrised H\'enon-like maps} with parametrisation $(f,\e)$. We will just write $F=(\phi,\pi_x)$ when the
parametrisation is not explicit. In the current setting we will simply call them H\'enon-like maps. We let $\H_\Omega(0)$ denote the subspace of the boundary of $\H_\Omega$ consisting of maps whose thickening is identically zero. We call such maps \emph{degenerate H\'enon-like maps}.

Observe that, for all $\Omega$, there is an imbedding $\uline{\i}\colon \U_{\Omega_x}\to\H_\Omega(0)$ given by
\[\uline{\i}\colon f\mapsto (f\circ\pi_x,\pi_x).\]
Therefore the renormalisation operator $\RU$ induces an operator on its image under $\uline{\i}$. A dynamical extension of this
operator was constructed in~\cite{Haz1}. More precisely:
\begin{thm}[see~\cite{Haz1}]\label{R-construction}
There are constants $C,\bar\e_0>0$ and a domain $\Omega=\Omega_x\times\Omega_y\subset\CC$, depending upon $\upsilon$, such that the following holds: 
for any $0<\bar\e<\bar\e_0$ there is a subspace $\H_{\Omega,\upsilon}(\bar\e)$ of $\H_{\Omega}(\bar\e)$ containing $\uline{\i}(\U_{\Omega_x,\upsilon})$ and a dynamically defined continuous operator 
\[\RH\colon \H_{\Omega,\upsilon}(\bar\e)\to \H_{\Omega}(C\bar\e^p)\subset \H_{\Omega}(\bar\e)\]
which is a continuous extension of $\uline{i}_*\RU$.
\end{thm}
This is called the \emph{H\'enon renormalisation operator}, or simply the \emph{renormalisation operator}, on $\H_{\Omega,\upsilon}(\bar\e)$.
\begin{rmk}\label{rmk:nonaffine}
As in the unimodal case $\RH$ is expressible as $\RH F=\MT^{-1}\circ \o{p}{F}\circ \MT$ where $\MT=\MT(F)\colon B\to B$. However $\MT$ is a non-affine coordinate change which is determined by the dynamics of $F$
(see~\cite{Haz1} for more details). This was required so that $\RH F$ again had a parametrisation.
\end{rmk}
\begin{thm}[see~\cite{Haz1}]\label{R-convergence}
There exists a $\bar\e_0>0$ such that for all $0<\bar\e<\bar\e_0$ the renormalisation operator
\[\RH\colon \H_{\Omega,\upsilon}(\bar\e)\to\H_{\Omega}(\bar\e)\]
has a unique fixed point $F_*$. Moreover $F_*=(f_*\circ\pi_x,\pi_x)$ where $f_*$ is the fixed point of $\RU$ and $F_*$ is hyperbolic with a codimension one stable manifold.
\end{thm}

\subsection{The Scope Maps}\label{subsect:scope}
As was noted in the above Remark~\ref{rmk:nonaffine}, the renormalisation $\RH F$ of $F$ is the \emph{non-affine} change of coordinates of the first return map of $F$ to a certain subdomain of $B$. This coordinate change $\MT=\MT(F)\colon B\to B$ is called the \emph{scope function}. In fact if we set $\oo{0}{\MT}=\MT$ and $\oo{w}{\MT}=\o{w}{F}\circ\MT$ for $w=1,\ldots,p-1$ then $\oo{w}{\MT}$ will be called the \emph{$w$-th scope function}, where $w\in W$. 

Now let $\I_{\Omega,\perm}(\bar\e)\subset \H_{\Omega,\perm}(\bar\e)$ denote the subspace of infinitely renormalisable H\'enon-like maps. Given $F\in\I_{\Omega,\upsilon}(\bar\e)$ we will denote the $n$-th renormalisation $\RH^n F$ by $F_n$. For $w\in W$ let $\oo{w}{\MT_n}=\o{w}{F_n}\circ\MT(F_n)\colon \Dom(F_{n+1})\to \Dom(F_{n})$ be the $w$-th scope function of $F_n$ as defined above, where $\Dom(F_n)$ denotes the domain of $F_n$. Then for $\word{w}{}=w_0\ldots w_n\in W^*$ the map
\[\MT^{\word{w}{}}=\MT_0^{w_0}\circ\ldots\circ \MT_n^{w_n}\colon \Dom(F_{n+1})\to \Dom(F_0)\]
is called the \emph{$\word{w}{}$-scope function}. Let $\uline\MT=\{\MT^{\word{w}{}}\}_{\word{w}{}\in W^*}$ denote the collection of all 'scope functions.

\subsection{The Renormalisation Cantor Set}
We define the \emph{renormalisation Cantor set}, $\Cantor=\Cantor(F)$, associated to $F$ by
\[\Cantor=\bigcap_{n\geq 0}\bigcup_{\word{w}{}\in W^n}\MT^{\word{w}{}}(B).\]
In~\cite{Haz1} a homeomorphism between $\Cantor$ and $\bar W$ was constructed which conjugates the action of $F$ with the action of addition by $1$ defined above. 
Let us denote the cylinder sets of $\Cantor$ under the action of $F$ by $\Cantor^{\word{w}{}}$. That is $\Cantor^{\word{w}{}}=\Cantor\cap \MT^{\word{w}{}}(B)$. 
Then the collection $\uline\Cantor=\{\Cantor^{\word{w}{}}\}_{\word{w}{}\in W^*}$ has the following structure
\begin{enumerate}
\item $F(\Cantor^{\word{w}{}})=\Cantor^{1+\word{w}{}}$ for all $\word{w}{}\in W^*$;
\item $\Cantor^{\word{w}{}}$ and $\Cantor^{\word{\tilde w}{}}$ are disjoint for all $\word{w}{}\neq \word{\tilde w}{}$ of the same length;
\item the disjoint union of the $\Cantor^{\word{w}{w}}$, $w\in W$, is equal to $\Cantor^{\word{w}{}}$, for all $\word{w}{}\in W^*$;
\item $\Cantor =\bigcup_{\word{w}{}\in W^n}\Cantor^{\word{w}{}}$ for each $n\geq 1$.
\end{enumerate} 
It was also shown each point $z\in\Cantor$ corresponds to a unique element $\word{w}{}$ of the infinite adding machine $\bar W$. We will call the word $\word{w}{}$ the \emph{address} of $z$. In particular
we define the \emph{tip} $\tau=\tau(F)$ to be the point in $\Cantor$ corresponding to the word $0^\infty$. In other words
\[\tau=\bigcap_{n\geq 1}\MT^{\word{0^n}{}}(B).\]
This is the point which in~\cite{dCML} and~\cite{Haz1} replaced the role of the critical value in the renormalisation theory for unimodal maps.

As the action of $F$ on $\Cantor$ is metrically isomorphic to the adding machine, $\Cantor$ has a unique $F$-invariant measure, $\mu$.
The \emph{Average Jacobian} $b=b(F)$ is then defined by
\begin{equation*} 
b(F)=\exp\int \log|\jac F|d\mu.
\end{equation*}
Now we can state the main result of~\cite{Haz1}.
\begin{thm}\label{thm:universality}
Given $F\in\I_{\Omega,\upsilon}(\bar\e_0)$ there exists a universal $a\in C^\omega(J,\RR)$ and universal $0<\rho<1$, depending upon $\upsilon,\Omega$ only, such that
\begin{equation}
F_n(x,y)=(f_n(x)-b^{p^n}a(x)y(1+\bigo(\rho^n)),x)
\end{equation}
where $f_n$ are unimodal maps converging exponentially to $f_*$, the fixed point of renormalisation of combinatorial type $\upsilon$.
\end{thm}
\subsection{The Induced Scope Maps and Cantor Sets}
For any $n>0$ we can construct the functions $\MT^{\word{w}{}}_n=\MT^{\word{w}{}}(F_n)$, the sets $\Cantor_n^{\word{w}{}}=\Cantor^{\word{w}{}}(F_n)$ and the points
$\tau_n=\tau(F_n)$ in exactly the same way
as we did above. The number $n$ is called the \emph{height} of $\MT^{\word{w}{}}_n$, $\Cantor_n^{\word{w}{}}$ and $\tau_n$ and the length of $\word{w}{}$ is called the
\emph{depth} of 
$\MT^{\word{w}{}}_n$ and $\Cantor_n^{\word{w}{}}$. 
Let $\uline\MT_n=\{\MT^{\word{w}{}}_n\}_{\word{w}{}\in W^*}$ and $\uline\Cantor_n=\{\Cantor_n^{\word{w}{}}\}_{\word{w}{}\in W^*}$.

\begin{rmk}
We use the terms height and depth to reflect a kind of duality in our construction, reflected in the issue of whether to call the $\MT_n$ telescope maps or microscope maps.
\end{rmk}

As the functions $\MT^{\word{0^{n-m}}{}}_m\colon \Dom(F_{n+1})\to\Dom(F_0)$ will be of particular importance we denote them by $\MT_{m,n}$. In~\cite{Haz1} the following two Propositions were proved.
\begin{prop}\label{F_n-estimate}
Given $F\in\I_{\Omega,\upsilon}(\bar\e)$ its renormalisations $F_n$ have the form 
\[F_n(z)=(\phi_{n}(z),\pi_x(z))\]
and the derivative of the maps $F_n$ have the form 
\[\D{F_n}{z}=\iibyii{\del_x\phi_{n}(z)}{\del_y\phi_{n}(z)}{1}{0}.\]
Moreover there exists constants a constant $C>0$ depending upon $\upsilon$ and $\Omega$ and a universal constant $0<\rho<1$ such that
\begin{enumerate}
\item
$|\del_x\phi_{n}|<C$;
\item
$C^{-1}b^{p^n}<|\del_y\phi_{n}|<Cb^{p^n}$.
\end{enumerate}
\end{prop}
An application of the Mean Value Theorem gives us the following.
\begin{lem}\label{MVT}
Given $F\in\I_{\Omega,\upsilon}(\bar\e)$ let $F_n$ denote its $n$-th renormalisation. 
For any $z_0,z_1\in \Dom (F_n)$ there exists $\xi, \eta,\in \brl z_0,z_1\brr$, the rectangle spanned by $z_0,z_1$, such that
\begin{align*}
\pi_x(F_{n}z_0)-\pi_x(F_{n}z_1)&=\del_x\phi_{n}(\xi)(\pi_x(z_0)-\pi_x(z_1))+\del_y\phi_{n}(\eta)(\pi_y(z_0)-\pi_y(z_1)) \\
\pi_y(F_{n}z_0)-\pi_y(F_{n}z_1)&=\pi_x(z_0)-\pi_x(z_1).
\end{align*}
\end{lem}
\begin{prop}\label{scopeestimate}
Given $F\in\I_{\Omega,\upsilon}(\bar\e)$ the scope maps $\MT_{m,n}=\MT_{m,n}(F)$ have the form 
\begin{equation}\label{eq:scope1}
\MT_{m,n}(z)=\tau_m+D_{m,n}\circ (\id +R_{m,n})(z-\tau_{n+1})
\end{equation}
where $D_{m,n}=\D{\MT_{m,n}}{\tau_{n+1}}$ is the derivative of $\MT_{m,n}$ at the $n$-th tip and $R_{m,n}=\R{\MT_{m,n}}{\tau_{n+1}}$ is a remainder term. More explicitly
\begin{equation}\label{eq:scope2}
\MT_{m,n}(z)=\tau_m+\sigma_{m,n}\ibyii{s_{m,n}(x+r_{m,n}(z-\tau_{n+1}))+t_{m,n}y}{y}
\end{equation}
where
\begin{equation}\label{eq:scope3}
D_{m,n}=\sigma_{m,n}\iibyii{s_{m,n}}{t_{m,n}}{0}{1}, \quad R_{m,n}(z)=\ibyii{r_{m,n}(z)}{0}.
\end{equation}
Moreover there is a constant $C>0$, depending upon $\Omega$ and $\upsilon$ and a universal analytic function $v_*\colon J\to\RR$ and universal constants\footnote{The constant
$a$ is actually $a=a(f_*(c_*))$ where $c_*$ is critical value of $f_*$ and $a(x)$ is the universal one-dimensional real analytic function given by Theorem~\ref{thm:universality}}
$a>0, 0<\rho,\sigma<1$, depending upon $\upsilon$ only, such that for any $0<m<n$ sufficiently large
\begin{enumerate}
\item\label{ineq:scope1}
\[\sigma^{n-m}(1-C\rho^m)<|\sigma_{m,n}|<\sigma^{n-m}(1+C\rho^m)\]
\item\label{ineq:scope2}
\[\sigma^{n-m}(1-C\rho^m)<|s_{m,n}|<\sigma^{n-m}(1+C\rho^m)\]
\item\label{ineq:scope3}
\[ab^{p^m}(1-C\rho^m)<|t_{m,n}|<ab^{p^m}(1+C\rho^m)\]
\item\label{ineq:scope4}
\[|x+r_{m,n}(x,y)-v_*(x)-c_{m}y^2|<C\rho^{n-m},\]
where $c_{m}=c_{m}(F)$ is a constant satisfying $c_{m}=\bigo(\bar\e^{p^m})$.
\end{enumerate}
\end{prop}
The quantities $\sigma_{m,n}, s_{m,n}$ and $t_{m,n}$ are called the \emph{scaling}, \emph{squeeze} and \emph{tilt}, respectively, for $\MT_{m,n}$.

\section{Boxings and Bounded Geometry}\label{boxing}
\subsection{Boxing the Cantor Set}
Let $F\in\I_{\Omega,\upsilon}(\bar\e_0)$ and let $\uline\Cantor$ and $\uline\MT$ be as in Section~\ref{prelim}. 
A collection of simply connected open sets $\uline{B}=\{B^{\word{w}{}}\}_{\word{w}{}\in W^*}$ is called a \emph{boxing} of $\uline\Cantor$ with respect to $F$ if 
\begin{b-enumerate}
\item $F(B^{\word{w}{}})\subset B^{1+\word{w}{}}$ for all $\word{w}{}\in W^*$,
\item $B^{\word{w}{}}$ and $B^{\word{\tilde{w}}{}}$ are disjoint for all $\word{w}{}\neq\word{\tilde{w}}{}$ of the same length,
\item the disjoint union of the $B^{\word{w}{w}}$, $w\in W$, is a subset of $B^{\word{w}{}}$, for all $\word{w}{}\in W^*$,
\item $\Cantor^{\word{w}{}}\subset B^{\word{w}{}}$ for all $\word{w}{}\in W^*$,
\end{b-enumerate}
The sets $B^{\word{w}{}}$ are called the \emph{pieces} of the boxing and the \emph{depth} of the piece $B^{\word{w}{}}$ is the length
of the word $\word{w}{}$. The scope functions give us a boxing $\uline B_{can}=\{B^{\word{w}{}}_{can}\}_{\word{w}{}\in W^*}$, where $B^{\word{w}{}}_{can}=\MT^{\word{w}{}}(B)$, which we will call the \emph{canonical boxing}.

Observe that the since the scope functions $\uline{\MT_n}=\{\oo{\word{w}{}}{\MT_n}\}_{\word{w}{}\in W^*}$ for $F_n$ can be written as $\oo{\word{w}{}}{\MT_n}=\MT_{0,n}^{-1}\circ\MT_{0,n}\circ\oo{\word{w}{}}{\MT_n}$ and $\MT_{0,n}\circ\oo{\word{w}{}}{\MT_n}\in \uline{\MT}$, the canonical boxing $\uline{B_{n,can}}$ for $F_n$ is the preimage under $\MT_{0,n}$ of all the pieces contained in $\MT_{0,n}(B)$. Hence the scope maps preserve the canonical boxings of various heights.

There is also another `standard' boxing, which we call the \emph{topological boxing}. The pieces are simply connected domains whose boundary consists of two arcs, one of which
is a segment of the unstable manifold of a particular periodic point and the other consisting of a segment of stable manifold of a different periodic point of the same period. 
These boxings in the period doubling case were first considered in~\cite{dCML} and extended  to arbitrary combinatorial types in~\cite{Haz1}.
\begin{defn}
We say that a boxing $\uline{B}=\{B^{\word{w}{}}\}_{\word{w}{}\in W^*}$ has \emph{bounded geometry} if there exist constants $C>1, 0<\kappa<1$ such that for all $\word{w}{}\in W^*,w,\tilde{w},\in W$,
\begin{align}
C^{-1}\dist(B^{\word{w}{w}},B^{\word{w}{\tilde{w}}})<\diam(B^{\word{w}{w}})<C\dist(B^{\word{w}{w}},B^{\word{w}{\tilde{w}}}) \label{geometry:necessary} \\
\kappa\diam(B^{\word{w}{}})<\diam(B^{\word{w}{w}})<(1-\kappa)\diam(B^{\word{w}{}}) \label{geometry:unnecessary}
\end{align}
We will say that $\Cantor$ has \emph{bounded geometry} if there exists a boxing $\uline B$ of $\uline\Cantor$ with bounded geometry. Otherwise we will say
$\Cantor$ has unbounded geometry.
\end{defn}
\begin{rmk} 
As the results we will prove are actually stronger than mere unbounded geometry. We will show that Property~\ref{geometry:necessary} is violated almost everywhere in one-parameter families of infinitely renormalisable H\'enon-like maps. We believe that any breakdown of Property~\ref{geometry:unnecessary} is much more dependent upon the choice of boxings - in principle we could take any boxing and just enlarge the one containing the tip. The only thing to show would then be whether the return of this box is contained in the original box. 
\end{rmk}
We will use the assumption below in the following sections for expositional simplicity. Its necessity will become clear in Section~\ref{construction} when we describe the construction. 
\begin{enumerate}
\item[(B-5)]\label{property:comparable} $B^{\word{w}{w}}\subset B^{\word{w}{}}_{can}$ for all $w\in W$ and all sufficiently large $\word{w}{}\in W^*$.
\end{enumerate}
This will allow us, given any boxing $\uline B$ of $\uline\Cantor$, to construct induced
boxings $\uline B_n$ at all sufficiently great heights. However below, in Lemma~\ref{lem:not-comparable}, we show this assumption is redundant.
\begin{lem}\label{lem:not-comparable}
Given a boxing $\uline B$ of $\uline\Cantor$ there is a boxing $\uline{\hat B}$ satisfying Property~(B-5) above such that if $\uline{\hat B}$ has unbounded geometry
then $\uline B$ has unbounded geometry. 
\end{lem}
\begin{proof}
Given a boxing $\uline B$ of $\uline\Cantor$ define $\uline{\hat B}$ to be the collection $\{\hat B^{\word{w}{}}\}_{w\in W^*}$ where
\[\hat B^{\word{w}{w}}=B^{\word{w}{w}}\cap B_{can}^{\word{w}{}}, \quad w\in W, \word{w}{}\in W^*\] 
It is clear that 
\[\dist(B^{\word{w}{}},B^{\word{\tilde{w}}{}})\leq \dist(\hat{B}^{\word{w}{}},\hat{B}^{\word{\tilde w}{}})\]
and
\[\diam(B^{\word{w}{}})\geq \diam(\hat{B}^{\word{w}{}}).\]
\end{proof}

\subsection{The Construction}\label{construction}
Now let us introduce the construction and set-up some notation that shall be used throughout the remainder of the paper.
Firstly, for any infinitely renormalisable H\'enon-like map, we
will change coordinates for each renormalisation so that the $n$-th tip, $\tau_n$, lies at the origin.
As this coordinate change is by translations only, this will not affect the geometry of the
Cantor set. The new scope maps will have the form
\[\afa\MT_{m,n}(z)=D_{m,n}\circ (\id+R_{m,n})(z).\]
Secondly, the following quantities will prove to be useful. Given $z=(x,y),\tilde z=(\tilde{x},\tilde{y})\in \Dom (F_{n+1})$ let
\begin{equation*}
\Upsilon_*(z,\tilde{z})=\frac{v_*(\tilde{x})-v_*(x)}{\tilde{y}-y},
\end{equation*}
where $v_*$ is the universal function given by Proposition~\ref{scopeestimate}.
Given $F\in \I_{\Omega,\upsilon}(\bar\e_0)$ and points $z,\tilde z\in B_{n+1}$ let
\begin{equation*}
\Upsilon_m(z,\tilde{z})=\Upsilon_*(z,\tilde{z})-c_{m}\frac{\tilde{y}^2-y^2}{\tilde{y}-y}
\end{equation*}
where $c_m=c_m(F)$ are the constants given by Proposition~\ref{scopeestimate}.
\begin{rmk}\label{rmk:signconvention}
A technicality that was not present in~\cite{dCML} is the following: the quantity $t_{m,n}/s_{m,n}$ (where $t_{m,n}$ and $s_{m,n}$ are tilt and the squeeze of $\MT_{m,n}$ as given by Proposition~\ref{scopeestimate})
is important in controlling horizontal overlap of pieces of a boxing. The sign of this will determine which boxes we take to ensure their images horizontally
overlap. Observe that the combinatorial type $\upsilon$ determines whether the sign of $t_{m,n}/s_{m,n}$ alternates or remains constant. This is due to the sign of $t_{m,n}$ being always negative, but the sign of $s_{i}$ will asymptotically depend upon the sign of the derivative of the presentation function at its fixed point so, as $s_{m,n}$ is the product of $s_i$, the sign of $s_{m,n}$ will either be $(1)^{n-m}$ or $(-1)^{n-m}$.
Consequently we will restrict ourselves to considering sufficiently large $m,n\in 2\NN$ or $2\NN+1$ to ensure $t_{m,n}/s_{m,n}$ is negative. Our method would also work for the other case, but this would require choosing more words and points below and doing a case analysis, which adds to the complications. 
\end{rmk}

\begin{defn}
Given words $\word{w}{},\word{\tilde{w}{}}$ the points $z_*^0,z_*^1\in\Cantor_*^{\word{w}{}}$, and $\tilde{z}_*^0\in\Cantor_*^{\word{\tilde{w}}{}}$ are \emph{well placed} if
\begin{enumerate}
\item\label{property:lineup1}
$x_*^0<x_*^1<\tilde{x}_*^0,\quad y_*^0<y_*^1<\tilde{y}_*^0$; 
\item\label{property:wellplaced1}
$\Upsilon_*(z_*^0,\tilde{z}_*^0)<\Upsilon_*(z_*^0,z_*^1)$.
\end{enumerate}
A pair of words $\word{w}{},\word{\tilde{w}{}}$ are called \emph{well chosen} if 
\begin{enumerate}
\item 
there exist well placed points $z_*^0,z_*^1\in\Cantor_*^{\word{w}{}}$, and $\tilde{z}_*^0\in\Cantor_*^{\word{\tilde{w}}{}}$;
\item 
$\word{w}{}$ and $\word{\tilde w}{}$ differ only on the last letter, i.e. $\word{w}{}=w_0\ldots w_{n-1}w_n$ and $\word{\tilde{w}}{}=w_0\ldots w_{n-1}\tilde{w}_n$ for
some $w_0,\ldots,w_n,\tilde{w}_n\in W$ and some integer $n>0$.
\end{enumerate}

\end{defn}
\begin{rmk}
Observe Property~(i) will occur for certain words as $\Cantor_*^{\word{w}{}}$ and $\Cantor_*^{\word{\tilde{w}}{}}$ are horizontally and
vertically separated if $\word{w}{}$ and $\word{\tilde{w}}{}$ have the same length. If the $t_{m,n}/s_{m,n}$ were positive we would change the ordering above.
\end{rmk}
\begin{lem}\label{lem:wellchosen}
Well chosen pairs of words exist.
\end{lem}
\begin{proof}
First we wish to find well-placed points, then it will become clear from our argument that we can assume they boxes with well chosen words. Recall that we have changed coordinates so that the tip $\tau_*$ lies at the origin. Let $\afa{f_*}$ denote the translation $f_*$ that agrees with this coordinate change. Observe that points in $\Cantor_*$ have the form $z=(\afa{f_*}(y),y)$ where $y$ lies in the one-dimensional Cantor attractor for $\afa{f_*}$ in the interval. Therefore given points $z_*^0,z_*^1, \tilde{z}_*\in\Cantor_*$ we have
\begin{equation}
\Upsilon_*(z_*^0,z_*^1)=\frac{v_*\circ\afa{f_*}(y_*^1)-v_*\circ\afa{f_*}(y_*^0)}{y^1_*-y^0_*}, \quad 
\Upsilon_*(z_*^0,\tilde{z}_*)=\frac{v_*\circ\afa{f_*}(\tilde{y}_*)-v_*\circ\afa{f_*}(y_*^0)}{\tilde{y}_*-y^0_*}.
\end{equation}
Since $v_*$ and $\afa{f_*}$ are analytic so is the function $v_*\circ\afa{f_*}$. Since the derivative of $v_*\circ\afa{f_*}$ is zero at the critical point $c_*$ analyticity implies there exists a neighbourhood $V$ around $c_*$ on which $v_*\circ\afa{f_*}$ is concave or convex. Therefore if $z_*^0,z_*^1, \tilde{z}_*\in\Cantor_*$ are any points whose $y$-projections lie in $V$ then Property~\ref{property:lineup1} implies Property~\ref{property:wellplaced1}, by the Mean Value Theorem for example. But choosing $y_*^0,y_*^1$ and $\tilde{y}_*$ to lie all either to the left of $c_*$ or to the right will give us Property~\ref{property:lineup1}.

Finally choosing the largest disjoint cylinder sets $\Cantor_*^\word{w}{}, \Cantor_*^{\word{\tilde{w}}{}}$ of $\Cantor_*$, of the same depth, such that $z_*^0,z_*^1\in\Cantor_*^\word{w}{}$ and $\tilde{z}_*\in\Cantor_*^\word{\tilde{w}}{}$ gives us the desired well-chosen words.
\end{proof}
We can now make the following assumptions. 
There exist words $\word{w}{},\word{\tilde{w}}{}$, of the same length, and points $z_*^0,z_*^1\in\Cantor_*^{\word{w}{}},
\tilde{z}_*^0,\tilde{z}_*^1\in\Cantor_*^{\word{\tilde{w}}{}}$, which we now fix, satisfying
\begin{enumerate}
\item\label{property:lineup2}
$x_*^0<x_*^1<\tilde{x}_*^0<\tilde{x}_*^1,\quad y_*^0<y_*^1<\tilde{y}_*^0<\tilde{y}_*^1$;
\item\label{property:wellplaced2}
the points $z_*^0, z_*^1,\tilde{z}_*^0$ are well placed.
\end{enumerate}
Given these points let us now define some quantities which shall prove to be useful. Let
\[\kappa_0=|\Upsilon_*(z_*^0,z_*^1)-\Upsilon_*(\tilde{z}_*^0,\tilde{z}_*^1)|, \quad \kappa_1=\frac{|y_*^1-y_*^0|}{|\tilde{y}_*^0-y_*^0|},\]
and
\[\kappa_2=|\tilde{y}_*^0-y_*^0|, \quad \kappa_3=|y_*^1-y_*^0|, \quad \kappa_4=|\tilde{y}_*^1-\tilde{y}_*^0|.\]
These are all well-defined nonzero quantities by Lemma~\ref{lem:wellchosen}.
For any $F\in\I_{\Omega,\upsilon}(\bar\e_0)$ let the points 
\[z_{n}^0=(x_{n}^0,y_{n}^0),z_{n}^1=(x_{n}^1,y_{n}^1)\in\Cantor_{n}^{\word{w}{}}\]
and
\[\tilde{z}_{n}^0=(\tilde{x}_{n}^0,\tilde{y}_{n}^0),\tilde{z}_{n}^1=(\tilde{x}_{n}^1,\tilde{y}_{n}^1)\in\Cantor_{n}^{\word{\tilde{w}}{}}\]
have the same respective addresses in $\Cantor_{n}$ (see subsection~\ref{subsect:unimodal} to recall the definition) as those of $z_*^0,z_*^1,\tilde{z}_*^0,\tilde{z}_*^1$ in $\Cantor_*$.
Let
\begin{equation}
M=\left[\frac{\Upsilon_*(z_*^0,\tilde{z}_*^0)-\frac{\kappa_1}{2}\Upsilon_*(z_*^0,z_*^1)}{1-\frac{\kappa_1}{2}},\Upsilon_*(z_*^0,\tilde{z}_*^0)\right].
\end{equation}
This is a well defined interval because $z_*^0, z_*^1$ and $\tilde{z}_*^0$ are well placed which implies $\Upsilon_*(z_*^0,z_*^1)>\Upsilon_*(z_*^0,\tilde{z}_*^0)$ and hence
\begin{equation}
\Upsilon(z_*^0,\tilde{z}_*^0)-\frac{\kappa_1}{2}\Upsilon_*(z_*^0,z_*^1)<\Upsilon_*(z_*^0,\tilde{z}_*^0)(1-\frac{\kappa_1}{2})
\end{equation}
Dividing by $1-\frac{\kappa_1}{2}$ and recalling $0<\kappa_1/2<1$ gives us the claim.
Fix a $\delta>0$ such that
\begin{equation}\label{deltaprop2}
M_\delta=\left[\frac{\Upsilon_*(z_*^0,\tilde{z}_*^0)-\frac{\kappa_1}{2}\Upsilon_*(z_*^0,z_*^1)}{1-\frac{\kappa_1}{2}}+\frac{\delta}{3}\left(\frac{3-\frac{\kappa_1}{2}}{1-\frac{\kappa_1}{2}}\right),\Upsilon_*(z_*^0,\tilde{z}_*^0)-\delta\right].
\end{equation}
is a well defined interval.
Choose $N>0$ sufficiently large so that
\begin{equation}\label{deltaprop3}
4C\rho^N<\frac{\kappa_2}{2}\left(1-\frac{\kappa_1}{2}\right)\frac{\delta}{3}
\end{equation}
and
\begin{equation}\label{prop6}
4C\rho^{N}(1/\kappa_3+1/\kappa_4)<\kappa_0/8.
\end{equation}
Let $\A\subset\I_{\Omega,\upsilon}(\bar\e_0)$ denote the subspace of all infinitely renormalisable H\'enon-like maps $F$ such that, for all $n>m>0, n-m>N$:
\begin{a-enumerate}
\item\label{prop1} 
$x_{n+1}^0<x_{n+1}^1<\tilde{x}_{n+1}^0<\tilde{x}_{n+1}^1,\quad y_{n+1}^0<y_{n+1}^1<\tilde{y}_{n+1}^0<\tilde{y}_{n+1}^1$;
\item\label{prop2} 
$1>|y_{n+1}^1-y_{n+1}^0|/|\tilde{y}_{n+1}^0-y_{n+1}^0|>\kappa_1/2$;
\item\label{prop4}
$|\tilde{y}_{n+1}^0-y_{n+1}^0|>\kappa_2/2, \ |y_{n+1}^1-y_{n+1}^0|>\kappa_3/2, \ |\tilde{y}_{n+1}^1-\tilde{y}_{n+1}^0|>\kappa_4/2$;
\item\label{prop3}
$|\Upsilon_m(z_{n+1}^0,z_{n+1}^1)-\Upsilon_m(\tilde{z}_{n+1}^0,\tilde{z}_{n+1}^1)|>\kappa_0/2$;
\item\label{prop5}
$|(x+r_{m,n}(z))-(v_*(x)-c_my^2)|<C\rho^{n-m}$ for all $z\in B_{n+1}$;
\item\label{deltaprop1}
$|\Upsilon_m(z_{n+1}^0,z_{n+1}^1)-\Upsilon_*(z_*^0,z_*^1)|, |\Upsilon_m(z_{n+1}^0,\tilde{z}_{n+1}^0)-\Upsilon_*(z_*^0,\tilde{z}_*^0)|<\delta/3$;
\item\label{deltaprop2}
$t_{m,n}/s_{m,n}<0$ and moreover
\[\left|\frac{t_{m,n}}{s_{m,n}}+a\frac{b^{p^m}}{\sigma^{n-m}}\right|<\delta/3;\]
\end{a-enumerate}
where $a,C,c_m, t_{m,n},s_{m,n},\sigma, v_*,\rho$ are the quantities described by Proposition~\ref{scopeestimate}. 

\begin{prop}
Given a family $F_b\in\I_{\Omega,\upsilon}(\bar\e_0)$ parametrised by the average Jacobian, there exists an integer $N_0>0$ and $0<b_0<1$ such that $\RH^n F_b\in \A$ for all
$n>N_0, 0\leq b\leq b_0$.
\end{prop}
\begin{proof}
This follows as $\RH^n(F_b)$ converges exponentially to $F_*$ which lies in $\A$, so we may choose the $N_0>0$ so that $\RH^n(F_0)\in \A$ for all $n>N_0$. Then it is clear there exists a $b_0>0$ such that $\RH^{N_0}(F_b)\in \A$ for all $0\leq b\leq b_0$ since $\A$ is open. It is also clear $\A$ is invariant under $\RH$ so the Proposition follows.\end{proof}

We now describe the construction. This was used in~\cite{dCML} and~\cite{Haz1} to prove several negative results, such as non-existence of continuous invariant line fields
(see these two references for further details). Let $F\in\A$ and let us fix $n,m\in 2\NN$ or $2\NN+1$ as per remark~\ref{rmk:signconvention} such that $n>m>0$ and $n-m>N$. Consider the maps $\MT_{0,m-1}, F_m, \MT_{m,n}$.
In reverse order, these map from height $n+1$ to height $m$, from height $m$ to itself and from height $m$ to height $0$ respectively (see figure~\ref{fig:perturb}). 

We will adopt the following notation convention: if we have a quantity $Q$ in the
domain of $\MT_{m,n}$ we will denote its images under $\MT_{m,n}, F_m$ and $\MT_{0,m-1}$ by $\dot Q, \ddot Q$ and $\dddot Q$ respectively.

\begin{figure}
\centering
\psfrag{amt}{\small $\MT_{0,m-1}$}
\psfrag{bf0}{\small $F_m$}
\psfrag{bmt}{\small $\MT_{m,n}$}
\psfrag{cf0}{\small $F_{n+1}$}
\psfrag{a}{$\Dom(F_0)$}
\psfrag{b}{$\Dom(F_m)$}
\psfrag{c}{$\Dom(F_{n+1})$}
\psfrag{ab0}{}
\psfrag{bb0}{}
\psfrag{bb1}{}
\psfrag{cb0}{}
\psfrag{cb1}{}

\includegraphics[width=1.0\textwidth]{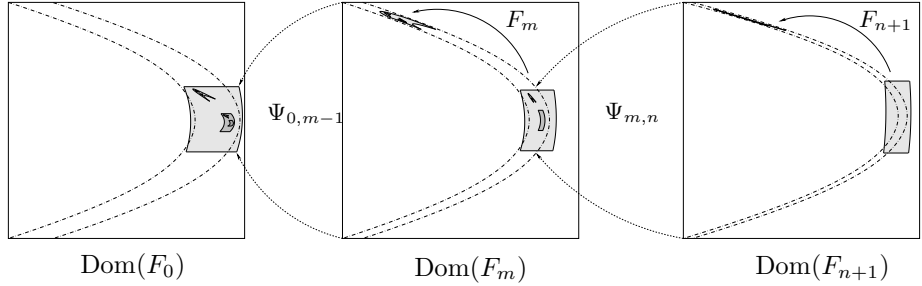}
\caption{perturbation of boxes near the tip}\label{fig:perturb}
\end{figure}






\section{Horizontal Overlapping Distorts Geometry}\label{horiz-overlap}
Recall that in the previous section we fixed well chosen words $\word{w}{}, \word{\tilde{w}}{}\in W^*$ with points $z_*^0,z_*^1\in\Cantor_*^\word{w}{}$ and
$\tilde{z}_*^0,\tilde{z}_*^1\in\Cantor_*^\word{\tilde{w}}{}$ so that $z_*^0,z_*^1$ and $\tilde{z}_*^0$ are well-placed. We make the following definition.
\begin{defn}
Given a boxing $\uline{B}$ we will say it satisfies property $\Hor_{\word{w}{},\word{\tilde{w}}{}}(m,n)$ if
the pieces $\oo{\word{w}{}}{B}_{n+1},\oo{\word{\tilde{w}}{}}{B}_{n+1}\in \uline{B}_{n+1}$ have images $\oo{0^{n-m}\word{w}{}}{B}_m, \oo{0^{n-m}\word{\tilde{w}}{}}{B}_m$, under $\MT_{m,n}$, which horizontally overlap.
\end{defn}
Throughout the rest of the section we will assume the boxing $\uline{B}$ is fixed.
\begin{lem}[Key Lemma]\label{key}
Given a constant $K>0$, there is a constant $C>0$ such that the following holds: given $F\in\I_{\Omega,\upsilon}(\bar\e_0)$, if there are points $z, \tilde z\in \Dom(F_{n+1})$ satisfying
\begin{itemize}
\item $|\pi_y(z)-\pi_y(\tilde{z})|>K$;
\item $|\pi_x(\dot{z})-\pi_x(\dot{\tilde{z}})|=0$;
\end{itemize}
then
\begin{equation}
|\Upsilon_*(z,\tilde{z})|-C\max(\rho^m,\rho^{n-m})<\frac{ab^{p^m}}{\sigma^{n-m}}<|\Upsilon_*(z,\tilde{z})|+C\max(\rho^m,\rho^{n-m})
\end{equation}
\end{lem}
\begin{proof}
Equality~(\ref{eq:scope2}) from Proposition~\ref{scopeestimate} tells us if $\dot{z}, \dot{\tilde{z}}$ lie on the same vertical line then
\begin{equation}
0=s_{m,n}(x+r_{m,n}(x,y)-\tilde{x}-r_{m,n}(\tilde{x},\tilde{y}))+t_{m,n}(y-\tilde{y}).
\end{equation}
Dividing by  $s_{m,n}(y-\tilde{y})$, which is nonzero, gives us
\begin{equation}
-\frac{t_{m,n}}{s_{m,n}}=\frac{[x+r_{m,n}(z)]-[\tilde{x}+r_{m,n}(\tilde{z})}{y-\tilde{y}}.
\end{equation}
By inequality~(\ref{ineq:scope4}) in Proposition~\ref{scopeestimate} implies
\begin{equation}\label{ineq:key1}
|\Upsilon_m(z,\tilde{z})|-\frac{C\rho^{n-m}}{|\tilde{y}-y|}<\left|\frac{t_{m,n}}{s_{m,n}}\right|<|\Upsilon_m(z,\tilde{z})|+\frac{C\rho^{n-m}}{|\tilde{y}-y|}.
\end{equation}
Again by inequality~(\ref{ineq:scope4}) in Proposition~\ref{scopeestimate} and the definition of $\Upsilon_m$ we know
\begin{equation}\label{ineq:key2}
|\Upsilon_*(z,\tilde{z})|-C\bar\e_0^{p^m}<|\Upsilon_m(z,\tilde{z})|<|\Upsilon_*(z,\tilde{z})|+C\bar\e_0^{p^m}.
\end{equation}
By inequalities~(\ref{ineq:scope2}) and~(\ref{ineq:scope3}) in Proposition~\ref{scopeestimate} we know there is a constant $C'>0$ such that
\begin{equation}\label{ineq:key3}
\left|\frac{t_{m,n}}{s_{m,n}}\right|(1-C'\rho^m)<\frac{ab^{p^m}}{\sigma^{n-m}}<\left|\frac{t_{m,n}}{s_{m,n}}\right|(1+C'\rho^m).
\end{equation}
Combining inequalities~(\ref{ineq:key1}), (\ref{ineq:key2}) and~(\ref{ineq:key3}), together with our first assumption and the observation $\bar\e_0^{p^m}=\bigo(\rho^m)$, gives us the result.
\end{proof}
\begin{cor}\label{cor:comparable}
There exists a constant $C>0$ such that the following holds:
let $F\in\I_{\Omega,\upsilon}(\bar\e)$ and let $z_{n+1}^0, \tilde{z}_{n+1}^0\in\Cantor_n$ have the same respective addresses as $z_*^0,\tilde{z}_*^0\in\Cantor_*$. 
If $|\pi_x(\dot{z}_{n+1}^0)-\pi_x(\dot{\tilde{z}}_{n+1}^0)|=0$ then
\begin{equation}
C^{-1}\sigma^{n-m}<b^{p^m}<C\sigma^{n-m}
\end{equation}  
\end{cor}
\begin{proof}
This follows as $z_{n+1}^0, \tilde{z}_{n+1}^0$ can be taken to be arbitrarily close to $z_*^0, \tilde{z}_*^0$ and so the constant $K>0$ in Lemma~\ref{key} will eventually only
depend upon the vertical distance between these points, which is fixed.
\end{proof}

\begin{prop}\label{h-overlap-close}
For any words $\word{w}{},\word{\tilde{w}}{}\in W^*$ there exists a $C_0>0$ such that the following holds:
for any $F\in\I_{\Omega,\upsilon}(\bar\e_0)$ and any boxing $\uline B$ of $F$, if points $z\in B_{n+1}^{\word{w}{}}$ and $\tilde{z}\in B_{n+1}^{\word{\tilde w}{}}$ 
satisfy $|\pi_x(\dot{z})-\pi_x(\dot{\tilde{z}})|=0$ then
\[\dist(\dddot{z},\dddot{\tilde{z}})<C_0\sigma^{2m}b^{p^m}\sigma^{n-m}.\]
\end{prop}
\begin{proof}
Let $z=(x,y), \tilde{z}=(\tilde{x},\tilde{y}), \dot{z}=(\dot{x},\dot{y}), \dot{\tilde{z}}=(\dot{\tilde{x}},\dot{\tilde{y}})$ and so on. Then by Proposition~\ref{scopeestimate}
and our hypothesis that $\dot{z},\dot{\tilde{z}}$ lie on the same vertical line, we know
\begin{align}
|\dot{\tilde{x}}-\dot{x}|&= 0 \\
|\dot{\tilde{y}}-\dot{y}|&= |\sigma_{m,n}||\tilde{y}-y|. \notag
\end{align}
Applying Lemma~\ref{MVT} we then know there exists $\eta\in\brl \dot{z},\dot{\tilde{z}}\brr$ such that
\begin{align}
|\ddot{\tilde{x}}-\ddot{x}|&= |\del_y\phi_m(\eta)| |\sigma_{m,n}||\tilde{y}-y| \\
|\ddot{\tilde{y}}-\ddot{y}|&= 0. \notag
\end{align}
Then Proposition~\ref{scopeestimate} once more implies
\begin{align}\label{eqn:xyparts}
|\dddot{\tilde{x}}-\dddot{x}|&= |\sigma_{0,m-1}||s_{0,m-1}||[\ddot{\tilde{x}}+r_{0,m-1}(\ddot{\tilde{z}})]-[\ddot{x}+r_{0,m-1}(\ddot{z})]| \\
|\dddot{\tilde{y}}-\dddot{y}|&= 0. \notag
\end{align}
But, by the Mean Value Theorem and that $\ddot{\tilde{y}}=\ddot{y}$, we find there is a $\xi\in [\ddot{x},\ddot{\tilde{x}}]$ such that
\begin{align}\label{eqn:meanval}
|[\ddot{\tilde{x}}+r_{0,m-1}(\ddot{\tilde{z}})]-[\ddot{x}+r_{0,m-1}(\ddot{z})]|&=|1+\del_xr_{0,m-1}(\xi,\ddot{y})||\ddot{\tilde{x}}-\ddot{x}| \\
&=|1+\del_xr_{0,m-1}(\xi,\ddot{y})||\del_y\phi_m(\eta)||\sigma_{m,n}||\tilde{y}-y|. \notag
\end{align}
It follows from Propositions~\ref{scopeestimate} and~\ref{F_n-estimate} that there are constants $C',C'',C'''>0$, independent of $m,n$, such that
\begin{equation}\label{eqn:threebounds}
|1+\del_xr_{0,m-1}(\xi,\ddot{y})|<C', \quad |\del_y\phi_m(\eta)|<C''b^{p^m}, \quad |\sigma_{m,n}|<C'''\sigma^{n-m}.
\end{equation}
Hence it follows from (\ref{eqn:xyparts}), (\ref{eqn:meanval}) and (\ref{eqn:threebounds}) that there is a $C_0>0$ such that
\[\dist(\dddot{z},\dddot{\tilde{z}})=|\dddot{\tilde{x}}-\dddot{x}|<C_0\sigma^{2m}b^{p^m}\sigma^{n-m}\]
\end{proof}

\begin{prop}\label{h-overlap-far}
For well chosen words $\word{w}{}$ and $\word{\tilde{w}}{}$ and points $z_*^0,z_*^1\in\Cantor_*^\word{w}{}$ and
$\tilde{z}_*^0,\tilde{z}_*^1\in\Cantor_*^\word{\tilde{w}}{}$ so that $z_*^0,z_*^1,\tilde{z}_*^0$ and $\tilde{z}_*^0,\tilde{z}_*^1,z_*^1$ are well-placed triples, there exists a
constant $C_1>0$, depending upon $\Omega,\upsilon$ and the above words and points only, such that the following holds:
Let $F\in\A$ and let $\uline B$ be a boxing for $F$. Then there exist points $z^0,z^1\in B_{n+1}^{\word{w}{}}, \tilde{z}^0,\tilde{z}^1\in B_{n+1}^{\word{\tilde w}{}}$ such that either
\[\dist(\dddot{z}_0,\dddot{z}_1)>C_1\sigma^m\sigma^{2(n-m)} \quad \mbox{or} \quad \dist(\dddot{\tilde{z}}_0,\dddot{\tilde{z}}_1)>C_1\sigma^m\sigma^{2(n-m)}.\] 
\end{prop}
\begin{proof}
Let $z^0=z_{n+1}^0, z^1=z_{n+1}^1$ and $\tilde{z}^0=\tilde{z}_{n+1}^0, \tilde{z}^1=\tilde{z}_{n+1}^1$.
By Proposition~\ref{scopeestimate}
\begin{align}
&|\dot{x}^1-\dot{x}^0| \\
&=|\sigma_{m,n}|\left| s_{m,n}([x^1+r_{m,n}(z^1)]-[x^0+r_{m,n}(z^0)])+t_{m,n}(y^1-y^0)\right| \notag
\end{align}
Applying Proposition~\ref{MVT} we then get
\begin{align}
&|\ddot{y}^1-\ddot{y}^0|=|\dot{x}^1-\dot{x}^0| \\
&=|\sigma_{m,n}|\left| s_{m,n}([x^1+r_{m,n}(z^1)]-[x^0+r_{m,n}(z^0)])+t_{m,n}(y^1-y^0)\right|. \notag
\end{align}
Then again applying Proposition~\ref{scopeestimate} we have
\begin{align}\label{y-dist}
&|\dddot{y}^1-\dddot{y}^0| \\
&=|\sigma_{0,m-1}||\sigma_{m,n}|\left|s_{m,n}([x^1+r_{m,n}(z^1)]-[x^0+r_{m,n}(z^0)])+t_{m,n}(y^1-y^0)\right|. \notag
\end{align}
By the same argument a similar expression holds for $|\dddot{\tilde{y}}^1-\dddot{\tilde{y}}^0|$.

It follows from Properties~(A-\ref{prop4}) that
\begin{equation}
|([x^1+r_{m,n}(z^1)]-[x^0+r_{m,n}(z^0)])-([v_*(x^1)+c_m(y^1)^2]-[v_*(x^0)+c_m(y^0)^2])|<2C\rho^{n-m}
\end{equation}
and
\begin{equation}
|([\tilde{x}^1+r_{m,n}(\tilde{z}^1)]-[\tilde{x}^0+r_{m,n}(\tilde{z}^0)])-([v_*(\tilde{x}^1)+c_m(\tilde{y}^1)^2]-[v_*(\tilde{x}^0)+c_m(\tilde{y}^0)^2])|<2C\rho^{n-m}.
\end{equation}
Then dividing by $|y^1-y^0|$ and applying~(A-\ref{prop5}) gives us
\begin{equation}
\left|\frac{[x^1+r_{m,n}(z^1)]-[x^0+r_{m,n}(z^0)]}{y^1-y^0}-\Upsilon_m(z^0,z^1)\right|<\frac{4C}{\kappa_3}\rho^{n-m}
\end{equation}
and similarly
\begin{equation}
\left|\frac{[\tilde{x}^1+r_{m,n}(\tilde{z}^1)]-[\tilde{x}^0+r_{m,n}(\tilde{z}^0)]}{\tilde{y}^1-\tilde{y}^0}-\Upsilon_m(\tilde{z}^0,\tilde{z}^1)\right|<\frac{4C}{\kappa_4}\rho^{n-m}.
\end{equation}
But by Properties~(A-\ref{prop3}) and~(\ref{prop6}) this implies
\[\kappa_0/4<\left|\frac{[x^1+r_{m,n}(z^1)]-[x^0+r_{m,n}(z^0)]}{y^1-y^0}-\frac{[\tilde{x}^1+r_{m,n}(\tilde{z}^1)]-[\tilde{x}^0+r_{m,n}(\tilde{z}^0)]}{\tilde{y}^1-\tilde{y}^0}\right|\]
and therefore either
\begin{equation}
\kappa_0/8<\left|\frac{[x^1+r_{m,n}(z^1)]-[x^0+r_{m,n}(z^0)]}{y^1-y^0}+\frac{t_{m,n}}{s_{m,n}}\right|
\end{equation}
or
\begin{equation}
\kappa_0/8<\left|\frac{[\tilde{x}^1+r_{m,n}(\tilde{z}^1)]-[\tilde{x}^0+r_{m,n}(\tilde{z}^0)]}{\tilde{y}^1-\tilde{y}^0}+\frac{t_{m,n}}{s_{m,n}}\right|
\end{equation}
or possibly both.

Now by Proposition~\ref{scopeestimate} there are constants $C',C'',C'''>0$ such that
\[|\sigma_{0,m-1}|>C'\sigma^m, \quad |\sigma_{m,n}|>C''\sigma^{n-m}, \quad |s_{m,n}|>C'''\sigma^{n-m}.\]
This, together with Property~(A-\ref{prop4}), equality~(\ref{y-dist}) and the estimate in the previous paragraph, implies there is a constant $C_1>0$ such that either
\[\dist(\dddot{z}^0,\dddot{z}^1)>C_1\sigma^m \sigma^{2(n-m)}\] 
or
\[\dist(\dddot{\tilde{z}}^0,\dddot{\tilde{z}}^1)>C_1\sigma^m \sigma^{2(n-m)}.\]
\end{proof}

We distill these three results into the following.
\begin{prop}\label{ifh-overlap}
For $\word{w}{},\word{\tilde{w}}{}\in W^*$ well chosen there exist constants $C_0,C_1>0$, depending upon $\upsilon,\Omega$ only, such that 
given $F\in\A$ the following holds:
for any boxing $\uline B$ with property $\Hor_{\word{w}{},\word{\tilde{w}}{}}(m,n)$ the pieces $B_0^{0^{m}10^{n-m}\word{w}{}},B_0^{0^{m}10^{n-m}\word{\tilde{w}}{}}\in
\uline{B}_0$ of depth $n+length(\word{w}{})$ satisfying
\[\dist(B_0^{0^{m}10^{n-m}\word{w}{}},B_0^{0^{m}10^{n-m}\word{\tilde{w}}{}})<C_0\sigma^{2m}b^{2p^m}\]
and
\[\diam(B_0^{0^{m}10^{n-m}\word{w}{}}) \ \mbox{or} \ \diam(B_0^{0^{m}10^{n-m}\word{\tilde{w}}{}}) >C_1\sigma^mb^{2p^m}\]
\end{prop}
\begin{proof}
Propositions~\ref{h-overlap-close} implies
\[\dist(B_0^{0^{m}10^{n-m}\word{w}{}},B_0^{0^{m}10^{n-m}\word{\tilde{w}}{}})<C_0\sigma^{m}b^{p^m}\sigma^{n-m},\]
while Proposition~\ref{h-overlap-far} implies one of
\[\diam(B_0^{0^{m}10^{n-m}\word{w}{}})>C_1\sigma^m \sigma^{2(n-m)},\quad \diam(B_0^{0^{m}10^{n-m}\word{\tilde{w}}{}})>C_1\sigma^m \sigma^{2(n-m)}.\]
is true. However Corollary~\ref{cor:comparable} implies $b^{p^m}$ and $\sigma^{n-m}$ are comparable. Hence the result follows.
\end{proof}
\begin{rmk}
Observe these bounds have no dependence upon $n$, the height at which the overlapping boxes `originate'. This suggests that only the overlapping distorts the geometry and not
that they are close to the tip, $\tau_m$, of $F_m$, which is a crucial part of our estimate. 
\end{rmk}

\section{A Condition for Horizontal Overlap}\label{universal}
Now we wish to show that this horizontal overlapping behaviour occurs sufficiently often.
Recall that in the previous section we fixed well chosen words $\word{w}{}, \word{\tilde{w}}{}\in W^*$ with points $z_*^0,z_*^1\in\Cantor_*^\word{w}{}$ and
$\tilde{z}_*^0,\tilde{z}_*^1\in\Cantor_*^\word{\tilde{w}}{}$ so that $z_*^0,z_*^1$ and $\tilde{z}_*^0$ are well-placed.
\begin{prop}\label{parameterh-overlap}
Given well chosen words $\word{w}{},\word{\tilde{w}}{}\in W^*$ with well placed points $z_*^0,z_*^1\in\Cantor_*^{\word{w}{}}, \tilde{z}\in\Cantor_*^{\word{\tilde{w}}{}}$ there
exist constants $0<A_0<A_1$, depending upon $\upsilon$ and $\Omega$ also, such that the following holds: given $F\in\A$ and any boxing $\uline B$,
if
\begin{equation}\label{universalinterval}\tag{\dag}
A_0<\frac{b_F^{p^m}}{\sigma^{n-m}}<A_1
\end{equation}
then property $\Hor_{\word{w}{},\word{\tilde{w}}{}}(m,n)$ is satisfied. That is, $B^{0^{n-m}\word{w}{}}_{m}$ and $B^{0^{n-m}\word{\tilde w}{}}_{m}$ horizontally overlap.
\end{prop}
\begin{proof}
Let $z^0=(x^0,y^0)=z_{n+1}^0, z^1=(x^1,y^1)=z_{n+1}^1$ and $\tilde{z}=(\tilde{x},\tilde{y})=\tilde{z}_{n+1}^0$. 
As we will take $m,n$ to be fixed integers for notational simplicity we also denote $\sigma_{m,n}, r_{m,n}, s_{m,n}, t_{m,n}, \Upsilon_m$ and $c_m$ by $\sigma, r, s, t, \Upsilon$ and $c$ respectively. We will still denote the limits of $\Upsilon_m$ and $c_m$ by $\Upsilon_*$ and $c_*$. 
Observe that $B_m^{0^{n-m}\word{w}{}}$ and $B_m^{0^{n-m}\word{\tilde{w}}{}}$ horizontally overlap if $\dot{x}^0<\dot{\tilde{x}}<\dot{x}^1$ or, equivalently,
\begin{equation}\label{overlap1}
0<\dot{\tilde{x}}^0-\dot{x}^0<\dot{x}^1-\dot{x}^0.
\end{equation}
\begin{figure}[htp]
\centering
\psfrag{z0}{$z^0$}
\psfrag{z1}{$z^1$}
\psfrag{zbar}{$\tilde{z}$}
\psfrag{z0dot}{$\dot{z}^0$}
\psfrag{z1dot}{$\dot{z}^1$}
\psfrag{zbardot}{$\dot{\tilde{z}}$}
\psfrag{amt}{$\MT_{m,n}$}
\psfrag{a}{$\Dom(F_m)$}
\psfrag{b}{$\Dom(F_{n+1})$}

\includegraphics[width=1.0\textwidth]{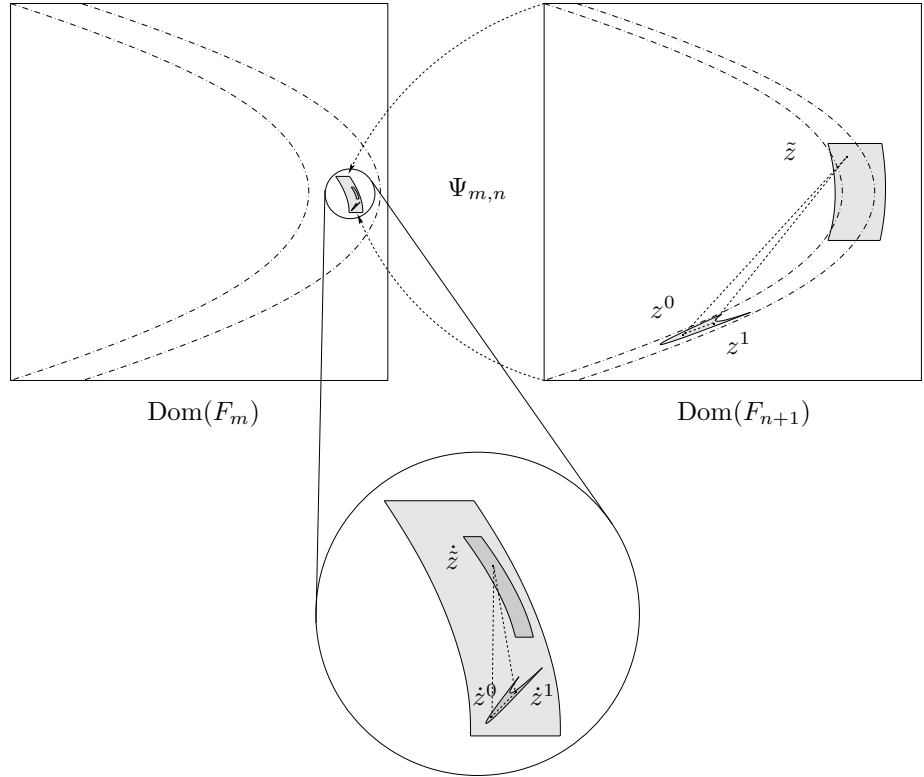}
\caption{overlapping pieces}\label{fig:overlap}
\end{figure}
\newpage
By Proposition~\ref{scopeestimate}, for $i=0,1$,
\begin{equation}
\dot{x}^i=\sigma(s[x^i+r(z^i)]+ty^i), \quad
\dot{\tilde{x}}=\sigma(s[\tilde{x}+r(\tilde{z})]+t\tilde{y})
\end{equation}
and therefore
\begin{align}
\dot{\tilde{x}}-\dot{x}^0&=\sigma(s([\tilde{x}+r(\tilde{z})]-[x^0+r(z^0)])+t(\tilde{y}-y^0)) \notag \\
\dot{x}^1-\dot{x}^0&=\sigma(s([x^1+r(z^1)]-[x^0+r(z^0)])+t(y^1-y^0)).
\end{align}
By Property~(A-\ref{prop5})
\begin{align}
\left| [\dot{\tilde{x}}-\dot{x}^0]-\sigma(s([v_*(\tilde{x})-v_*(x^0)]+c[(\tilde{y})^2-(y^0)^2])+t(\tilde{y}-y^0))\right|&<2C\sigma s\rho^{n-m} \notag \\
\left| [\dot{x}^1-\dot{x}^0]-\sigma(s([v_*(x^1)-v_*(x^0)]+c[(y^1)^2-(y^0)^2])+t(y^1-y^0))\right|&<2C\sigma s\rho^{n-m}.
\end{align}
Hence sufficient conditions for~(\ref{overlap1}) to hold are
\begin{equation}
0<\sigma(s([v_*(\tilde{x})-v_*(x^0)]+c[(\tilde{y})^2-(y^0)^2]+t(\tilde{y}-y^0))-2C\sigma s\rho^{n-m}
\end{equation}
and
\begin{align}
&\sigma(s([v_*(\tilde{x})-v_*(x^0)]+c[(\tilde{y})^2-(y^0)^2])+t(\tilde{y}-y^0)) \notag \\
&<\sigma(s([v_*(x^1)-v_*(x^0)]+c[(y^1)^2-(y^0)^2])+t(y^1-y^0))-4\sigma s\rho^{n-m}.
\end{align}
Since, by hypothesis, $\sigma, s>0$, and by hypothesis~(A-\ref{prop1}) we know $\tilde{y}-y>0$, dividing both of these inequalities by $\sigma s(\tilde{y}-y)$ and applying hypothesis~(A-\ref{prop4}) gives us
\begin{equation}
\frac{4C\rho^{n-m}}{\kappa_2}<\frac{2C\rho^{n-m}}{\tilde{y}-y}<\Upsilon(\tilde{z},z^0)+\frac{t}{s}
\end{equation}
and
\begin{equation}
\Upsilon(\tilde{z},z^0)+\frac{t}{s}<\frac{\kappa_1}{2}\left(\Upsilon(z^1,z^0)+\frac{t}{s}\right)-\frac{4C\rho^{n-m}}{\tilde{y}-y}<\frac{\kappa_1}{2}\left(\Upsilon(z^1,z^0)+\frac{t}{s}\right)-\frac{8C\rho^{n-m}}{\kappa_2}.
\end{equation}
Hence if
\begin{equation}\label{overlap2-a}
\frac{4C\rho^{n-m}}{\kappa_2}<\Upsilon(\tilde{z},z^0)+\frac{t}{s}
\end{equation}
and
\begin{equation}\label{overlap2-b}
\Upsilon(\tilde{z},z^0)+\frac{t}{s}<\frac{\kappa_1}{2}+\frac{t}{s}<\frac{\kappa_1}{2}\left(\Upsilon(z^1,z^0)+\frac{t}{s}\right)-\frac{8C\rho^{n-m}}{\kappa_2}
\end{equation}
then~(\ref{overlap1}) is satisfied and so there is horizontal overlap. Now let us show that there exists constants $0<A_0<A_1$ such that~(\ref{universalinterval}) implies inequalities~(\ref{overlap2-a}) and~(\ref{overlap2-b}). Let us treat inequality~(\ref{overlap2-a}) first. We claim that
\begin{equation}\label{condition-1}
\frac{ab^{p^m}}{\sigma^{n-m}}<\Upsilon_*(\tilde{z}_*,z^0_*)-\delta
\end{equation}
implies~(\ref{overlap2-a}). By Property~(A-\ref{deltaprop2})
\begin{equation}
\left|\frac{t}{s}\right|<\frac{ab^{p^m}}{\sigma^{n-m}}+\frac{\delta}{3}
\end{equation}
and by Property~(A-\ref{deltaprop1})
\begin{equation}
\Upsilon_*(\tilde{z}_*,z^0_*)<\Upsilon(\tilde{z},z^0)+\frac{\delta}{3}.
\end{equation}
Combining these gives us
\begin{equation}
\left|\frac{t}{s}\right|<\Upsilon(\tilde{z},z^0)-\frac{\delta}{3}.
\end{equation}
By hypothesis~(\ref{deltaprop3}) and Property~(A-\ref{prop2}) we know $\frac{8C\rho^{n-m}}{\kappa_2}<\frac{\delta}{3}$. Hence
\begin{equation}
\left|\frac{t}{s}\right|<\Upsilon(\tilde{z},z^0)-\frac{8C\rho^{n-m}}{\kappa_2}.
\end{equation}
Finally recall that $t/s<0$, so multiplying by $-1$ and reversing the above inequality gives~(\ref{overlap2-a}) as required.

Next we claim that
\begin{equation}\label{condition-2}
\frac{\Upsilon_*(\tilde{z}_*,z^0_*)-\frac{\kappa_1}{2}\Upsilon_*(z^1_*,z^0_*)}{1-\frac{\kappa_1}{2}}+\frac{\delta}{3}\frac{2}{1-\frac{\kappa_1}{2}}<\frac{ab^{p^m}}{\sigma^{n-m}}-\frac{\delta}{3}
\end{equation}
implies inequality~(\ref{overlap2-b}). 
From~(\ref{deltaprop3}) we know that $\frac{8C\rho^{n-m}}{\kappa_2\left(1-\frac{\kappa_1}{2}\right)}<\frac{\delta}{3}$ and from Property~(A-\ref{deltaprop1}) we know
\begin{align}
\frac{\Upsilon(\tilde{z},z^0)-\frac{\kappa_1}{2}\Upsilon(z^1,z^0)}{1-\frac{\kappa_1}{2}}
&< \frac{[\Upsilon_*(\tilde{z}_*,z^0_*)+\frac{\delta}{3}]-\frac{\kappa_1}{2}[\Upsilon_*(z^1_*,z^0_*)-\frac{\delta}{3}]}{1-\frac{\kappa_1}{2}} \\
&=\frac{\Upsilon_*(\tilde{z}_*,z^0_*)-\frac{\kappa_1}{2}\Upsilon_*(z^1_*,z^0_*)}{1-\frac{\kappa_1}{2}}+\frac{\delta}{3}\left(\frac{1+\frac{\kappa_1}{2}}{1-\frac{\kappa_1}{2}}\right)
\end{align}
Together these imply
\begin{equation}\label{overlap2-b-innerbound1}
\frac{\Upsilon(\tilde{z},z^0)-\frac{\kappa_1}{2}\Upsilon(z^1,z^0)}{1-\frac{\kappa_1}{2}}+\frac{8C\rho^{n-m}}{\kappa_2\left(1-\frac{\kappa_1}{2}\right)}
<\frac{\Upsilon_*(\tilde{z}_*,z^0_*)-\frac{\kappa_1}{2}\Upsilon_*(z^1_*,z^0_*)}{1-\frac{\kappa_1}{2}}+\frac{\delta}{3}\frac{2}{1-\frac{\kappa_1}{2}}.
\end{equation}
By Property~(A-\ref{deltaprop2}) we know
\begin{equation}\label{overlap2-b-innerbound2}
\frac{ab^{p^m}}{\sigma^{n-m}}-\frac{\delta}{3}<\left|\frac{t}{s}\right|
\end{equation}
so the above two inequalities~(\ref{overlap2-b-innerbound1}) and~(\ref{overlap2-b-innerbound2}) imply
\begin{equation}
\frac{\Upsilon(\tilde{z},z^0)-\frac{\kappa_1}{2}\Upsilon(z^1,z^0)}{1-\frac{\kappa_1}{2}}+\frac{8C\rho^{n-m}}{\kappa_2\left(1-\frac{\kappa_1}{2}\right)}
<\left|\frac{t}{s}\right|.
\end{equation}
Since $1-\frac{\kappa_1}{2}>0$, this is equivalent to
\begin{equation}
\Upsilon(\tilde{z},z^0)-\frac{\kappa_1}{2}\Upsilon(z^1,z^0)<\left|\frac{t}{s}\right|(1-\frac{\kappa_1}{2})-\frac{8C\rho^{n-m}}{\kappa_2}.
\end{equation}
Recalling that $t/s<0$ then tells us
\begin{equation}
\frac{t}{s}(1-\frac{\kappa_1}{2})+\frac{8C\rho^{n-m}}{\kappa_2}<\frac{\kappa_1}{2}\Upsilon(z^1,z^0)-\Upsilon(\tilde{z},z^0).
\end{equation}
which, upon rearranging, gives us
\begin{equation}
\Upsilon(\tilde{z},z^0)+\frac{t}{s}+\frac{8C\rho^{n-m}}{\kappa_2}<\frac{\kappa_1}{2}\left(\Upsilon(z^1,z^0)+\frac{t}{s}\right)
\end{equation}
which, by moving the error term to the right of the inequality sign gives us~(\ref{overlap2-b}) as required.

Finally set
\begin{align}
A_0&=a^{-1}\left[\left(\frac{\Upsilon_*(\tilde{z}_*,z^0_*)-\frac{\kappa_1}{2}\Upsilon_*(z^1_*,z^0_*)}{1-\frac{\kappa_1}{2}}\right)+\frac{\delta}{3}\left(\frac{3-\frac{\kappa_1}{2}}{1-\frac{\kappa_1}{2}}\right)\right] \\
A_1&=a^{-1}\left[\Upsilon_*(\tilde{z}_*,z^0_*)-\delta\right].
\end{align}
The interval $[A_0,A_1]$ is well defined by Property~(A-\ref{deltaprop2}).
Then inequality~(\ref{universalinterval}) implies, since $a>0$, together with~(\ref{condition-1}) and~(\ref{condition-2}) that inequalities~(\ref{overlap2-a}) and inequality~(\ref{overlap2-b}) hold and therefore the boxes overlap.
\end{proof}

\section{Construction of the Full Measure Set}\label{fullmeas}
We will now prove the following result.
\begin{thm}\label{thm:fullmeas}
Given any $0<A_0<A_1$, $0<\sigma<1$ and any $p\geq 2$ the set of parameters $b\in [0,1]$ for which there are infinitely many $0<m<n$ satisfying
\begin{equation}\label{universalinterval}\tag{\dag}
A_0<\frac{b^{p^m}}{\sigma^{n-m}}<A_1
\end{equation}
is a dense $G_\delta$ set with full Lebesgue measure.
\end{thm} 
\begin{rmk}
We note that this result is purely analytical; it has no dynamical content and as such is quite separate from the other sections.
\end{rmk}
We introduce the following notation, setting
\[d=n-m; \quad \delta=\delta(m)=1/p^m;\quad \alpha_i=\log A_i/\log\sigma=\log_\sigma A_i.\]
and letting $I_{d,\delta}$ be the set of $b$ which satisfy inequality~(\ref{universalinterval}). That is
\[I_{d,\delta}=\left[\sigma^{d\delta}A_0^{\delta},\sigma^{d\delta}A_1^{\delta}\right].\]
The following two lemmas are an easy calculation and are left to the reader.
\begin{lem}\label{lem:diam-dist}
\begin{enumerate}
\item $\diam (I_{d,\delta})=\sigma^{d\delta}(A_1^\delta-A_0^\delta)$.
\item If $I_{d+1,\delta}, I_{d,\delta}$ are disjoint then $I_{d+1,\delta}$ lies to the left of $I_{d,\delta}$.
\item If $I_{d',\delta'}, I_{d,\delta}$ are disjoint and $I_{d',\delta'}$ lies to the left of $I_{d,\delta}$ then
\[\dist(I_{d,\delta},I_{d',\delta'})=\sigma^{d\delta}(A_0^\delta-\sigma^{d'\delta'-d\delta}A_1^{\delta'}).\]
\end{enumerate}
\end{lem}
\begin{rmk}
In the proof of Proposition~\ref{G-thm} we will see there is a dichotomy: either, for a fixed $\delta>0$,  $I_{d,\delta}, I_{d+1,\delta}$ are always disjoint or they always intersect, for all
$d>0$, and moreover if property holds for one $\delta$ then it also holds for every choice of $\delta$. This depends on whether $A_1\sigma<A_0$ holds or not.
\end{rmk}
\begin{lem}\label{lem:d_maxmin}
Let $I_{d,\delta}, I_{d',\delta'}, I_{d'',\delta''}$ be pairwise disjoint and assume $I_{d',\delta'}$ lies to the left of $I_{d,\delta}$. Then $I_{d'',\delta''}$
lies to the right of $I_{d',\delta'}$ when
\[d''\leq\frac{\delta'}{\delta''}(d'+\alpha_1)-\alpha_0\]
and $I_{d'',\delta''}$ lies to the left of $I_{d,\delta}$ when
\[d''\geq\frac{\delta}{\delta''}(d+\alpha_0)-\alpha_1.\]
\end{lem}
\begin{lem}
Suppose $I_{d,\delta}, I_{d',\delta'}$ are disjoint and $I_{d',\delta'}$ lies to the left of $I_{d,\delta}$. Let $0<\delta''<\min (\delta,\delta')$. Let
$d_{min}''\leq d''\leq d_{max}''$ be the range of all $d''$ for which $I_{d'',\delta''}$ lies strictly between $I_{d,\delta}$ and $I_{d',\delta'}$. If the $I_{d'',\delta''}$ are
pairwise disjoint then
\begin{align*}
\left|\bigcup_{d''=d_{min}''}^{d_{max}''}I_{d'',\delta''}\right|
&=(A_1^{\delta''}-A_0^{\delta''})\frac{\sigma^{d''_{min}\delta''}-\sigma^{(d''_{max}+1)\delta''}}{1-\sigma^{\delta''}}
\end{align*}
\end{lem}
\begin{proof}
If the $I_{d'',\delta''}$ are pairwise disjoint then
\begin{equation}
\left|\bigcup_{d''=d_{min}''}^{d_{max}''}I_{d'',\delta''}\right|=\sum_{d''=d''_{min}}^{d''_{max}}|I_{d'',\delta''}|.
\end{equation}
Consequently, Lemma~\ref{lem:diam-dist} and the summation formula for geometric series implies the result.
\end{proof}
\begin{rmk}
By Lemma~\ref{lem:d_maxmin} we know that $d''_{max}$ and $d''_{min}$ have the form
\[d''_{max}=\lfloor\frac{\delta'}{\delta''}(d'+\alpha_1)-\alpha_0\rfloor;\quad d''_{min}=\lceil\frac{\delta}{\delta''}(d+\alpha_0)-\alpha_1\rceil.\]
\end{rmk}
We will also need the following Lemma.
\begin{lem}\label{lem:quotient}
Let $\sigma,P,Q\in \RR$ satisfy $0<\sigma\leq 1$ and $0<P<Q$. Then there exists a a positive real number $\bar s>0$ such that for all $0<s<\bar s$ we have
\[\half<\frac{\sigma^s P-\sigma^{-s}Q}{P-Q}.\]
\end{lem}
\begin{lem}\label{lem:gapfilling}
Assume $\sigma A_1<A_0$. Then there exists a constant $0<L\leq 1$ such that the following holds: choose any admissable $\delta,\delta', d,d'>0$ such that $I_{d,\delta}$, and $I_{d',\delta'}$ are disjoint and $I_{d',\delta'}$ lies to the left of $I_{d,\delta}$.
Then there exists a $\bar\delta<\delta,\delta'$ such that for any admissable $0<\delta''=\delta(m'')<\bar\delta$,
\[L\dist(I_{d,\delta},I_{d',\delta'})<\sum_{d''=d''_{min}}^{d''_{max}}|I_{d'',\delta''}|.\]
Moreover we can take $L=\quarter\left|\frac{1}{\log\sigma}\right|\left(1-\frac{A_0}{A_1}\right)\leq 1$.
\end{lem}
\begin{proof}
First observe that
\[\dist(I_{d,\delta},I_{d',\delta'})=A_0^\delta\sigma^{d\delta}-A_1^{\delta'}\sigma^{d'\delta'}\]
and
\[\sum_{d''=d''_{min}}^{d''_{max}}|I_{d'',\delta''}|=(A_1^{\delta''}-A_0^{\delta''})\frac{\sigma^{d''_{min}\delta''}-\sigma^{(d''_{max}+1)\delta''}}{1-\sigma^{\delta''}}.\]
We wish to approximate this last quantity. By Lemma~\ref{lem:d_maxmin} we know that
\[\delta(d+\alpha_0)-\alpha_1\delta''<d_{min}''\delta''<\delta(d+\alpha_0)-\alpha_1\delta''+\delta''\]
and
\[\delta'(d'+\alpha_1)-\delta''\alpha_0<(d_{max}''+1)\delta''<\delta'(d'+\alpha_1)-\delta''\alpha_0+\delta''.\]
Hence
\[A_0^\delta\sigma^{\delta d}\frac{\sigma^{\delta''}}{A_1^{\delta''}}-A_1^{\delta'}\sigma^{\delta'd'}\frac{1}{A_0^{\delta''}}
<\sigma^{d''_{min}\delta''}-\sigma^{(d''_{max}+1)\delta''}<
A_0^\delta\sigma^{\delta d}\frac{1}{A_1^{\delta''}}-A_1^{\delta'}\sigma^{\delta'd'}\frac{\sigma^{\delta''}}{A_0^{\delta''}}.\]
We also know, by the Mean Value Theorem and the concavity of $x\mapsto x^\delta$ for $\delta<1$, that
\[\delta''A_1^{\delta''-1}\frac{A_1-A_0}{1-\sigma^{\delta''}}<\frac{A_1^{\delta''}-A_0^{\delta''}}{1-\sigma^{\delta''}}<\delta''A_0^{\delta''-1}\frac{A_1-A_0}{1-\sigma^{\delta''}}.\]
Together these imply
\begin{equation}\label{ineq:gapfilling-lb1}
K\left(A_0^\delta\sigma^{\delta d}\sigma^{\delta''}-A_1^{\delta'}\sigma^{\delta'd'}\frac{A_1^{\delta''}}{A_0^{\delta''}}\right)<\sum_{d''=d''_{min}}^{d''_{max}}|I_{d'',\delta''}|
\end{equation}
where
\[K=K(\delta'')=\left(1-\frac{A_0}{A_1}\right)\left(\frac{\delta''}{1-\sigma^{\delta''}}\right).\]
Now observe that $\sigma A_1<A_0$ implies
\[A_0^\delta\sigma^{\delta d}\sigma^{\delta''}-A_1^{\delta'}\sigma^{\delta'd'}\sigma^{-\delta''}
<A_0^\delta\sigma^{\delta d}\sigma^{\delta''}-A_1^{\delta'}\sigma^{\delta'd'}\frac{A_1^{\delta''}}{A_0^{\delta''}}.\]
Therefore Lemma~\ref{lem:quotient} tells us, substituting $A_0^\delta\sigma^{\delta d}, A_1^{\delta'}\sigma^{\delta'd'}$ and $\delta'$ for $P,Q$ and $s$ respectively, there exists a constant $\delta_0>0$ such that for all $\delta''<\delta_0$,
\begin{equation}\label{ineq:gapfilling-lb2}
\frac{1}{2}<\frac{A_0^\delta\sigma^{\delta d}\sigma^{\delta''}-A_1^{\delta'}\sigma^{\delta'd'}(A_1/A_0)^{\delta''}}{A_0^\delta\sigma^{\delta d}-A_1^{\delta'}\sigma^{\delta'd'}}.
\end{equation}
Also observe that, by l'Hopital's rule,
\[\lim_{\delta''\to 0}\frac{\delta''}{1-\sigma^{\delta''}}=\lim_{\delta''\to 0}-\frac{1}{\sigma^{\delta''}\log\sigma}=\left|\frac{1}{\log\sigma}\right|,\]
and hence there exists a constant $\delta_1>0$ such that for all $\delta''<\delta_1$
\begin{equation}\label{ineq:gapfilling-lb3}
K(\delta'')=\left(1-\frac{A_0}{A_1}\right)\left(\frac{\delta''}{1-\sigma^{\delta''}}\right)>\frac{1}{2}\left|\frac{1}{\log\sigma}\right|\left(1-\frac{A_0}{A_1}\right).
\end{equation}
Therefore, if we let $\bar\delta=\min_{i=0,1}\delta_i$, inequalities~(\ref{ineq:gapfilling-lb2}) and~(\ref{ineq:gapfilling-lb3}) tell us that for any $\delta''<\bar\delta$,
\begin{equation}
\quarter\left|\frac{1}{\log\sigma}\right|\left(1-\frac{A_0}{A_1}\right)\dist(I_{d,\delta},I_{d',\delta'})< K(\delta'')\left(A_0^\delta\sigma^{\delta d}\sigma^{\delta''}-A_1^{\delta'}\sigma^{\delta'd'}\frac{A_1^{\delta''}}{A_0^{\delta''}}\right).
\end{equation}
Therefore by inequality~(\ref{ineq:gapfilling-lb1}) the Proposition follows.
\end{proof}
\begin{prop}\label{G-thm}
There exists a dense $G_\delta$ subset of $[0,b_0]$ with full relative Lebesgue measure such that each point lies in infinitely many $I_{d,\delta}$.
\end{prop}
\begin{proof}
There are two cases. The first is when $A_1\sigma\geq A_0$. Then
\[(A_0\sigma^{d+1})^{\delta}<(A_0\sigma^d)^{\delta}\leq(A_1\sigma^{d+1})^{\delta}<(A_1\sigma^d)^\delta,\]
that is, the right endpoint of $I_{d+1,\delta}$ lies to the right of the left endpoint of $I_{d,\delta}$. Therefore $I_{d+1,\delta}$ and $I_{d,\delta}$ overlap for all $d,\delta>0$. Hence
for each point $x\in (0,b_0]$ and any admissible $\delta>0$  there exists an integer $d=d(x,\delta)>0$ such that $x\in I_{d(x,\delta),\delta}$. Therefore $x$ lies in infinitely many $I_{d,\delta}$ and clearly $(0,b_0]$ is a dense $G_\delta$ with full relative Lebesgue measure in $[0,b_0]$.

The second case is when $A_1\sigma<A_0$. Then observe that $I_{d+1,\delta}$ and $I_{d,\delta}$ will be pairwise disjoint for all $d,\delta>0$. For any such pair let $J_{d,\delta}=[(\sigma^{d+1}A_1)^\delta,(\sigma^dA_0)^\delta]$ denote the
corresponding gap. The idea is to construct an infinite sequence of full measure sets, each a countable union of intervals $I_{d,\delta}$. We do this by the following inductive process. For a given $\delta$ we take the union of all $I_{d,\delta}$, this gives us gaps which we fill with $I_{d',\delta'}$, which leads to further gaps and so on. We can fill these gaps by a definite amount each time by Lemma~\ref{lem:gapfilling}. Hence the resulting set will have full Lebesgue measure.

Now let us proceed with the proof. First let us introduce the following notation. Given a union $T\subset [0,b_0]$ of disjoint intervals we will denote by $T_\delta$ the union of all $J_{d,\delta}$ strictly contained in $T$. We will use the notation $T_{\delta,\delta'}=(T_\delta)_{\delta'}$, $T_{\delta,\delta'\delta''}=(T_{\delta,\delta'})_{\delta''}$, and so on. We will denote the complement of $T_{\delta,\delta',\ldots}$ by $S_{\delta,\delta',\ldots}$.

Let $0<b_1<b_0$. We will show that there is a dense $G_\delta$ subset of full relative Lebesgue measure in $[b_1,b_0]$ with the required properties and then send $b_1$ to zero. Therefore let $T=[b_1,b_0]$. Let $\Delta=\{\delta(m)\}_{m\in\NN}$ denote the set of all admissible $\delta$'s ordered decreasingly. Let us construct an infinite subset $\Delta_0$ of $\Delta$ with infinite complement as follows. First choose $\delta_0^{(0)}$ to be arbitrary. Assume $\Delta_0^{(n)}=\{\delta_0,\ldots,\delta^{(n)}\}$ is given. Then Lemma~\ref{lem:gapfilling} tells us there is a $\delta>0$ such that for any $\delta_0^{(n+1)}<\delta$,
\[|T_{\delta_0,\ldots,\delta_0^{(n)}, \delta_0^{(n+1)}}|<(1-L_0)|T_{\delta_0,\ldots, \delta_0^{(n)}}|.\]
where $L_0$ is the contraction constant given by the same Lemma. We may do this as there are only finitely many gaps in $T_{\delta_0,\ldots, \delta_0^{(n-1)}}$. It is clear that by this process we can choose the $\Delta_0^{(n)}$ such that their limit $\Delta_0$ has complement with infinite cardinality. Also observe that, inductively
\[|T_{\delta_0,\ldots,\delta_0^{(n)}, \delta_0^{(n+1)}}|<(1-L_0)^{n+1}|T|,\]
so the limiting set $T_0$ will have zero measure since $0<L_0<1$. Hence its complement, $S_0$, which is a dense countable union of open intervals by construction, will have full relative Lebesgue measure.

Now assume we are given pairwise disjoint subsets $\Delta_0,\ldots,\Delta_N\subset\Delta$ whose union has infinite cardinality and we have the subsets $T_0,\ldots T_N\subset T$. Construct $\Delta_{N+1}=\{\delta_{N+1}^{(n)}\}_{n\in\NN}\subset\Delta$ disjoint from all these sets such that
\begin{equation}\label{contraction}
|T_{\delta_{N+1},\ldots,\delta_{N+1}^{(n-1)}, \delta_{N+1}^{(n)}}|<(1-L_0)|T_{\delta_{N+1},\ldots, \delta_{N+1}^{(n)}}|
\end{equation}
for all $n>0$ and such that the union of $\Delta_0,\ldots,\Delta_N,\Delta_{N+1}$ has complement with infinite cardinality. We can do this by the same argument as in the preceding paragraph. Also by the preceding paragraph it is clear that $T_{N+1}=\lim_{n\to\infty}T_{\delta_{N+1}^{(0)},\ldots,\delta_{N+1}^{(n)}}$ has zero measure and its complement $S_{N+1}$ is a dense countable union of open intervals with full relative Lebesgue measure. Therefore we construct a sequence of subsets $S_0,\ldots,S_n,\ldots\subset T$ which are dense countable unions of open intervals with full relative Lebesgue measure, implying their common intersection $S=\bigcup_{n\geq 0}S_n$ is a dense $G_\delta$ with full relative Lebesgue measure. 

Now let us show that any $x\in S$ is contained in infinitely many $I_{d,\delta}$'s. For each $n\geq 0$, $x$ is contained $S_n$. But $S_n$ is the union of $I_{d,\delta}$'s with $\delta\in\Delta_n$ and so $x$ lies in one of these. Since the $\Delta_n$ are pairwise disjoint, if $x\in I_{d_n,\delta_n}\cap I_{d_m,\delta_m}$ for $\delta_n\in\Delta_n,\delta_m\in\Delta_m,m\neq n$ then $\delta_n\neq\delta_m$. Hence $x$ is contained in infinitely many $I_{d,\delta}$'s.
\end{proof}

\section{Proof of the Main Theorem}\label{mainproof}
All the result so far have been for individual maps $F\in\I_{\Omega,\upsilon}(\bar\e_0)$. We will need the following lemma to make these statements about single maps applicable
to one parameter families parametrised by $b$.
\begin{lem}\label{lem:into-A}
Let $F_b\in \I_{\Omega,\upsilon}(\bar\e_0)$ be a one parameter family parametrised by the average Jacobian $b=b(F_b)\in [0,b_0)$. Then there is an $N>0$ and $0<b_1<b_0$ such that $\RH^NF_b\in \A$
for all $b\in [0,b_1]$. 
\end{lem}
\begin{proof}
The set $\A$ is an open neighbourhood of $F_*$ in the closure of $\H_{\Omega}$. We know that $\dist(\RH^nF_b,F_*)<\rho^n\dist(F_b,F_*)$, where $\dist$ denotes the adapted metric. 
Therefore there is an $N>0$ such that $\RH^nF_b\in \A$ for all integers $n>N$.
\end{proof}
We are now in a position to prove Theorem~\ref{main}.
\begin{proof}
By Lemma~\ref{lem:into-A} there is an integer $N>0$ and a $b_1>0$ such that $\RH^nF_b\in\A$ for all $n>N, b\in [0,b_1]$. Let $\tilde F_b=\RH^NF_b$.

Proposition~\ref{parameterh-overlap} implies if $\tilde F_b\in A$ then
for every $b$ satisfying inequality~(\ref{universalinterval}), $\tilde F_b$ has property $\Hor_{\word{w}{},\word{\tilde{w}}{}}(m,n)$. By Theorem~\ref{thm:fullmeas} the set, $\tilde S$, of parameters $b$ for which $\Hor_{\word{w}{},\word{\tilde{w}}{}}(m,n)$ is satisfied
for infinitely many $m,n$ has full Lebesgue measure. But then by Proposition~\ref{ifh-overlap} if $b$ lies in this set then $\tilde F_b$ has unbounded geometry.

Now we retrieve the statement for $F_b$ as follows. First observe that mapping $\Cantor(\tilde F_b)$ under $\MT_{0,N}(F_b)$ we get a subset of $\Cantor(F_b)$. The maps
$\MT_{0,N}(F_b)$ have bounded distortion by Proposition~\ref{scopeestimate}. Hence if $\Cantor(\tilde F_b)$ has unbounded geometry so will $\Cantor(F_b)$. Secondly we need to
show
\[S\subset \{b:\Cantor(\tilde F_b) \ \mbox{has unbounded geometry}\}\]
is a dense $G_\delta$ with full relative Lebesgue measure. This follows as $b(\tilde F_b)= b^{p^N}$, but $b\mapsto b^{p^N}$ preserves these properties, so by comparability
and injectivity the map $b(F_b)\mapsto b(\tilde F_b)$ must also preserve these properties. Since $\tilde S$ is a dense $G_\delta$ with full relative Lebesgue measure $S$ must also.
\end{proof}

\bibliographystyle{plain}
\bibliography{henon2}
\end{document}